\documentclass[11pt,twoside, a4paper, english, reqno]{amsart}
\usepackage[dvips]{epsfig}
\usepackage{amscd}
\usepackage{amssymb}
\usepackage{amsthm}
\usepackage{amsmath}
\usepackage{latexsym}

\usepackage{upref}
\usepackage{hyperref}
\usepackage{color}
\usepackage[usenames,dvipsnames]{xcolor}

\theoremstyle{plain}
\newtheorem{thm}{Theorem}[section]
\theoremstyle{plain}
\newtheorem{lem}[thm]{Lemma}
\newtheorem{prop}[thm]{Proposition}

\theoremstyle{definition}
\newtheorem{defi}{Definition}[section]
\newtheorem{rem}{Remark}[section]

\newenvironment{Assumptions}
{
\setcounter{enumi}{0}

\begin{enumerate}}
{\end{enumerate} }

\newcommand{\R}{\ensuremath{\mathbb{R}}}

\newcommand{\goto}{\ensuremath{\rightarrow}}
\newcommand{\grad}{\ensuremath{\nabla}}
\newcommand{\eps}{\ensuremath{\varepsilon}}

\numberwithin{equation}{section} \allowdisplaybreaks

\title[On stochastic optimal control of evolutionary $p$-Laplace equation]
{Stochastic optimal control of a evolutionary $p$-Laplace equation with multiplicative L\'{e}vy noise.}

\date{}



\author[Ananta K. Majee]{Ananta K. Majee}
\address[Ananta K. Majee]{\newline
Department of Mathematics, Indian Institute of Technology Delhi,
Hauz Khas, New Delhi-110016, India.}
\email[]{majee@maths.iitd.ac.in}

\thanks{}

\begin{document}
\begin{abstract}
In this article, we are interested in an initial value optimal control problem for a evolutionary $p$-Laplace
equation driven by multiplicative L\'{e}vy noise. We first present wellposedness of a weak solution by using an implicit time discretization of
the problem, along with the Jakubowski version of the Skorokhod theorem for a non-metric space. We then formulate associated control problem, and establish existence of an optimal solution by using variational 
method and exploiting the convexity property of the cost functional. 
\end{abstract}

\maketitle

\section{Introduction}
The last couple of decades have witnessed remarkable advances on the larger area
of stochastic partial differential equations that are driven by L\'{e}vy noise. An worthy
reference on this subject is \cite{peszat}. In this article, we are interested in the specific problem of evolution equation with L\'{e}vy noise, and aim 
to prove existence of a weak optimal solution of an initial value control evolutionary $p$-Laplace equation driven  by L\'{e}vy noise. A formal description of our problem as follows. 
Let $\big(\Omega, \mathcal{F}, \mathbb{P}, \{\mathcal{F}_t\}_{t\ge 0} \big)$ be a filtered probability space satisfying the usual hypotheses \textit{i.e.}
  $\{\mathcal{F}_t\}_{t\ge 0}$ is a right-continuous filtration such that $\mathcal{F}_0$ contains all the $\mathbb{P}$-null subsets of $(\Omega, \mathcal{F})$. In addition, let
   $N({\rm d}z,{\rm d}t)$ be a time homogeneous Poisson random measure \footnote[1]{ For the definition of a time homogeneous Poisson random measure, we refer to see \cite[pp. 631]{erika2009}.} on $\R$ with intensity measure $m(\,{\rm d}z)$ with respect to the same stochastic basis.
 We are interested in the initial value control evolution equation of the type 
 \begin{equation}\label{eq:p-laplace}
 \begin{aligned}
  {\rm d}u -\mbox{div}_x\big(|\grad u|^{p-2}\grad u + \vec{f}(u)\big)\,{\rm d}t &= \int_{|z|>0}\eta(u;z)\widetilde{N}(dz,{\rm d}t) \quad \text{in}\,\,\Omega \times D_T, \\
  u&=0 \quad \text{on}\,\, \Omega \times \partial D_T\,, \\
  u(0,\cdot)&=u_0(\cdot) + U(\cdot) \quad \text{in}\,\,\Omega \times  D\,,
  \end{aligned}
 \end{equation}
 where $p>2$, $D_T=(0,T)\times D$ with $T>0$ fixed, $D\subset \R^d$ is a bounded domain with Lipschitz boundary $\partial D$, $\partial D_T=(0,T)\times \partial D$, $u$ is the unknown random 
 scalar valued function, $\vec{f}:\R\goto \R^d$ is a given flux function, and $ \widetilde{N}({\rm d}z,{\rm d}t)= N({\rm d}z,{\rm d}t)-\, m({\rm d}z)\,{\rm d}t $, the compensated time homogeneous Poisson random measure. Furthermore, $(u,z)\mapsto \eta(u; z)$
 is a real valued function defined on the domain $ \R\times \R$. The stochastic integral in the right hand side of \eqref{eq:p-laplace} is defined
 in the L\'{e}vy-It\^{o} sense.
 \vspace{.1cm}
 
 We point out that adding a Brownian component to the L\'{e}vy noise term on the right hand side of 
 \eqref{eq:p-laplace} would make it more general, and the results of this paper are still valid under appropriate conditions.
\vspace{.2cm}

The equation \eqref{eq:p-laplace} could be viewed as a stochastic perturbation of a evolution $p$-Laplacian equation with nonlinear sources.
Equations of this type arise in the field of mechanics, physics and biology \cite{dibenedetto, wu}. In the case $\eta=0$, the equation \eqref{eq:p-laplace}
becomes a deterministic evolution $p$-Laplacian equation with nonlinear sources, and there is a plethora literature (see \cite{nabana,zhao} and references 
therein) for its wellposedness. 
\vspace{.1cm}

Due to more technical novelties, the study of wellposedness result in case of nonlinear $p$-evolutionary equation with nonlinear stochastic forcing is more subtle. The presence of nonlinearity in the drift and 
diffusion terms in equation prevents us to define a semi-group solution. Moreover, because of nonlinear perturbation ${\rm div}_x f(u)$ of $p$-Laplace operator~$(p>2)$, one can not use the results of monotone or 
locally monotone SPDEs, see e.g., \cite{Liu, Pardoux}. In a recent article \cite{zimmermann}, the authors have considered \eqref{eq:p-laplace} with cylindrical Wiener process $W=\{W_t: t\in [0,T]\}$
in $L^2(D)$, and proved wellposedness of strong solution. In \cite{zimmermann}, existence of a martingale solution is shown by constructing an approximate solution 
(via implicit time discretization) and deriving its {\it  a-priori} estimates which are used to apply Jakubowski-Skorokhod theorem in a non-metric space. Then, using an argument
of path-wise uniqueness and Gy\"{o}ngy-Krylov characterization \cite{gyongy-krylov} of convergence in probability, the authors established wellposedness of strong solution.  
\vspace{.1cm}

In this paper, our goal is to find a weak admissible solution $\pi^*:=\big( \Omega^*, \mathcal{F}^*, \mathbb{P}^*,\{ \mathcal{F}_t^*\},  N^*, u^*, U^*\big)$ which minimizes
\begin{equation} \label{eq:control-problem}
 \begin{aligned}
  \mathcal{J}(\pi)= \mathbb{E}\Big[ \int_0^T \|u(t)-u_{\rm tar}(t)\|_{L^2(D)}^2\,{\rm d}t + \|U\|_{W^{1,p}}^p + \Psi(u(T))\Big] \\
  \text{with}~~\pi=\big( \Omega, \mathcal{F},\mathbb{P},\{ \mathcal{F}_t\},  N, u, U\big)~~\text{subject to}~~\eqref{eq:p-laplace}\,,
 \end{aligned}
\end{equation}
for a given deterministic target profile $u_{\rm tar}$, and terminal payoff $\Psi$.  The existing literature (see e.g. \cite{Nagase}) on stochastic optimal control with SPDEs
mainly considers those which has a mild solution, which is not available for problem \eqref{eq:p-laplace}. For this reason, we use variational method to construct a minimizer $\pi^*$ of \eqref{eq:control-problem}, 
see also \cite{Serrano, TMAV-2019}. Being motivated from \cite{ Serrano, TMAV-2019,zimmermann}, our aim is twofold:

\begin{itemize}
 \item [i)] Firstly, we prove existence of a weak solution of the problem 
\eqref{eq:p-laplace}. We construct an approximate solutions $\tilde{u}_{\Delta t}:=\big\{ \tilde{u}_{\Delta t}(t); t\in [0,T]\}$ (cf.~\eqref{eq:defi-piecewise-affine-function}) via implicit time 
discretization, and derive its {\it a-priori} bounds which is used to show the tightness of the laws of sequence $(\tilde{u}_{\Delta t})$, denoted by $\mathcal{L}(\tilde{u}_{\Delta t})$, in some appropriate 
space via Aldous condition (see Definition \ref{defi:aldous-condition}). We then use the Jakubowski version of the Skorokhod theorem in a non-metric space to show existence of
a weak solution of \eqref{eq:p-laplace}. We use smooth approximation of absolute value function and then apply It\^{o}-L\'{e}vy formula, and pass to the limit as approximation parameter goes to zero to 
show the path-wise uniqueness of weak solutions. 
\item[ii)] Secondly, we construct a minimizer $\pi^*$ of the control problem \eqref{eq:control-problem} by considering a minimizing weak admissible solutions
$\pi_n=\big( \Omega_n, \mathcal{F}_n,\mathbb{P}_n,\{ \mathcal{F}_t^n\},  N_n, u_n, U_n\big)$ along with Skorokhod's theorem and exploiting the convexity property of the cost functional $\mathcal{J}$ with respect to
the control variable. 
\end{itemize}
\vspace{.1cm}

The remaining of the paper is organized as follows. We state the assumptions, detail the technical framework and state the main results
in Section \ref{sec:technical-framework}. In Section \ref{sec:existence-weak-solu}, we construct an approximate solutions,
derive its {\it a-priori} estimates, show the tightness of the laws of the approximate solutions in some space, and then apply
the Jakubowski version of the Skorokhod theorem to have a existence of a weak solution of the problem. Moreover, 
path-wise uniqueness of weak solutions is shown in  Subsection \ref{subsec:uniqueness}. The final section is devoted to establish existence of an optimal solution of the initial value control problem 
\eqref{eq:control-problem}.
 \section{Technical framework and statement of the main results}\label{sec:technical-framework}
 Throughout this paper, we use the letter $C$ to denote various generic constants.  In the sequel, we denote by $\big\langle \cdot,\cdot\big\rangle$,
 the pairing between $W_0^{1,p}(D)$ and $W^{-1,p^\prime}(D)$ where $p^\prime$ denotes the convex conjugate of $p$.
 \vspace{.1cm}
 
 In the theory of stochastic evolution equations, two types of solution concept are considered namely strong solution and weak solution. A strong solution is typically 
 an analytically weak solution (in space) on a given stochastic basis. In general, for a nonlinear non-Lipschitz drift operator, one may not able to prove existence of a strong solution, 
 and therefore needs to consider concept of weak solution.
 
 \begin{defi}\label{defi:weak-solun}(Weak solution)
 A weak solution of \eqref{eq:p-laplace} is a $7$-tuple $\bar{\pi}=\big( \bar{\Omega}, \bar{\mathcal{F}}, \bar{\mathbb{P}}, \{\bar{\mathcal{F}}_t\},  
 \bar{N}, \bar{u}, \bar{U}\big)$ such that 
 \begin{itemize}
  \item [i)] $(\bar{\Omega}, \bar{\mathcal{F}},\bar{\mathbb{P}})$ is a complete probability space endowed with the filtration 
  $\{\bar{\mathcal{F}}_t\}$ satisfying the usual hypotheses.
  \item[ii)] $\bar{N}$ is a time-homogeneous Poisson random measure on $\R$ with intensity measure $m(dz)$ with respect to the filtration 
  $\{\bar{\mathcal{F}}_t\}$.
  \item[iii)] $\bar{U}$ is measurable with $\bar{\mathbb{P}}$-a.s. $\omega \in \bar{\Omega}$, $\bar{U}(\omega,\cdot)\in W^{1,p}(D)$.
  \item[iv)] $\bar{u}: \bar{\Omega}\times [0,T]\rightarrow L^2(D)$ is an $W_0^{1,p}(D)$-valued $\{\bar{\mathcal{F}}_t\}$-predictable stochastic process such that $\bar{\mathbb{P}}$-a.s.,
  \begin{itemize}
   \item [a)] $\bar{u}\in L^p\big(0,T; W_0^{1,p}(D)\big)\cap L^\infty\big(0,T; L^2(D)\big)$ and $\bar{u}(0,\cdot)=u_0 + \bar{U}$ in $L^2(D)$.
   \item[b)] for all $t\in [0,T]$, there holds 
   \begin{align*}
   \bar{u}(t)= u_0 + \bar{U} + \int_{0}^t \mbox{div}_x \big( |\nabla \bar{u}|^{p-2}\nabla \bar{u} + \vec{f}(\bar{u})\big)\,{\rm d}s + 
   \int_0^t \int_{|z|>0}\eta(\bar{u};z)\widetilde{\bar{N}}({\rm d}z,{\rm d}s)  \quad \text{in}\quad L^2(D). 
  \end{align*}
  \end{itemize}
 \end{itemize}
 \end{defi}
 We show the wellposedness of weak solution of \eqref{eq:p-laplace}, in the sense of Definition \ref{defi:weak-solun}, under the following assumptions: 
\begin{Assumptions}
\item \label{A1} $u_0 \in L^2(D)$.
\item \label{A2}  $\vec{f}: \R \mapsto \R^d$ is $C^2$ and Lipschitz continuous with $\vec{f}(0)=0$.
 
\item \label{A3} $\eta(0;z)= 0$ for all $z\in \R$. Moreover, there exists positive constant $0<\lambda^* <1$ such that \footnote[2]{Here we denote $ x\wedge y:= \min\{x,y\}$.}
$ ~\text{for all}~  u,v \in \R$ and $z\in \R$
 \begin{align*}
 | \eta(u;z)-\eta(v;z)|  & \leq  \lambda^* |u-v|(1\wedge |z|).
 \end{align*}
\item \label{A4}
The L\'{e}vy measure $m({\rm d}z)$ is a Radon measure on $\R \setminus \{0\}$ with a possible singularity at $z=0$, which satisfies 
\begin{align*}
 c_{\eta}:= \int_{|z|>0} (1\wedge |z|^2)\,m({\rm d}z) < + \infty. 
\end{align*}
\end{Assumptions}
\begin{thm}\label{thm:existence-weak}
Let the assumptions \ref{A1}-\ref{A4} be true. Let $\big(\Omega, \mathcal{F}, \mathbb{P}, \{\mathcal{F}_t\} \big)$ be a given
filtered probability space satisfying the usual hypotheses and $N$ be a time-homogeneous Poisson random measure on $\R$ with intensity measure $m({\rm d}z)$
defined on $\big(\Omega, \mathcal{F}, \mathbb{P},\{\mathcal{F}_t\}\big)$, and $U\in L^p(\Omega; W^{1,p}(D))$. Then there exists a unique weak solution
$\bar{\pi}=\big( \bar{\Omega}, \bar{\mathcal{F}}, \bar{\mathbb{P}},  \{\bar{\mathcal{F}}_t\}, 
 \bar{N}, \bar{u}, \bar{U}\big)$ of the problem \eqref{eq:p-laplace}
in the sense of Definition \ref{defi:weak-solun}. Moreover
\begin{itemize}
 \item [i)] $\mathcal{L}(U)=\mathcal{L}(\bar{U})$ on $W^{1,p}(D)$ with $\bar{U}\in L^p(\bar{\Omega}; W^{1,p}(D))$, and
 \item[ii)] there exists a positive constant $C>0$ such that
\begin{align}
 \bar{\mathbb{E}}\Big[\sup_{0\le t\le T} \|\bar{u}(t)\|_{L^2(D)}^2 + \int_0^T \|\bar{u}(t)\|_{W_0^{1,p}(D)}^p\,{\rm d}t \Big] \le C \,\bar{\mathbb{E}}\Big[ \|u_0\|_{L^2(D)}^2 + \|U\|_{W^{1,p}(D)}^p\Big]\,.
 \label{esti:bound-weak-solun}
\end{align}
\end{itemize}
\end{thm}
\begin{rem}
 We remark that the assumption ~\ref{A3} is natural in the context of L\'{e}vy noise with the exception of $\lambda^*\in (0,1)$, which is necessary to handle the nonlocal nature of the It\^{o}-L\'{e}vy formula
 for the path-wise uniqueness in Subsection \ref{subsec:uniqueness}; see also \cite[Remark $1$]{BKM-2015}. Finally, the assumptions \ref{A1}-\ref{A4} collectively ensures existence of a weak solution of
 the problem \eqref{eq:p-laplace}. 
\end{rem}
\begin{rem}
 In view of \cite[Remark $1.1$]{zimmermann}, $\bar{\mathbb{P}}$-a.s., $\bar{u}\in C_w\big([0,T]; L^2(D)\big)$\footnote[3]{For any Banach space $\mathbb{X}$, $~~C_w([0, T ]; \mathbb{X})$ denotes the
Bochner space of weakly continuous functions with values in $\mathbb{X}$.} and  hence for all $t\in [0,T]$, $u(t,\cdot)$ is a stochastic process with values in $L^2(D)$.
\end{rem}
We denote by $\mathcal{U}_{\rm ad}^w(u_0;T)$ the set of weak admissible solutions to the problem \eqref{eq:p-laplace} in the sense Theorem \ref{thm:existence-weak}. Then, the initial value control problem 
\eqref{eq:control-problem} can be re-written as follows.
\begin{defi}\label{defi:optimal-control}
 Let the assumptions of Theorem \ref{thm:existence-weak} hold, and let $\Psi$ be a given Lipschitz continuous function on $L^2(D)$ and $u_{\rm tar} \in L^p(0,T; W_0^{1,p}(D))$ be a given deterministic target profile. A weak optimal solution of \eqref{eq:control-problem} is a $7$-tuple
 $\pi^*:=\big( \Omega^*, \mathcal{F}^*, \mathbb{P}^*,\{ \mathcal{F}_t^*\}, N^*, u^*, U^*\big)\in \mathcal{U}_{\rm ad}^w(u_0;T)$ such that
 \begin{align}
  \mathcal{J}(\pi^*)= \inf_{\pi \in \mathcal{U}_{\rm ad}^w(u_0; T)} \mathcal{J}(\pi):=\Lambda\,. \label{eq:optimal-control}
 \end{align}
\end{defi}
The associated control $U^*$ in $\pi^*$, as in \eqref{eq:optimal-control} is called weak optimal control of the control problem \eqref{eq:control-problem}.
\begin{thm}\label{thm:existence-optimal-control}
 There exists a weak optimal solution $\pi^*$ of \eqref{eq:control-problem} in the sense of Definition \ref{defi:optimal-control}.
\end{thm}

\section{Wellposedness of weak solution}\label{sec:existence-weak-solu}
In this section, we establish wellposedness of a weak solution for \eqref{eq:p-laplace}. To do this, we first construct an approximate solution
via implicit time discretization scheme, and then derive necessary uniform bounds.
\subsection{Implicit Euler scheme}
For $N\in \mathbb{N}^*$, let $0=t_0< t_1< \cdots < t_N=T$ be a uniform partition of $[0,T]$ with mesh size $\Delta t:=\frac{T}{N}>0$ i.e., $t_k=k\Delta t$ for $0\le k\le N$.
\vspace{.1cm}

For any $u_0\in L^2(D)$, there exists a sequence $\{ u_{0,\Delta t}\}\in W_0^{1,p}(D)$ such that $u_{0,\Delta t}\goto u_0$ in $L^2(D)$ as $\Delta t \goto 0$. Moreover, the following estimate holds:
\begin{align}
 \frac{1}{2}\|u_{0,\Delta t}\|_{L^2(D)}^2 + \Delta t\, \|\nabla u_{0,\Delta t}\|_{L^p(D)}^p \le \frac{1}{2}\|u_0\|_{L^2(D)}^2\,. \label{esti:initial-approx}
\end{align}
For its proof, we refer to see \cite[Lemma $30$]{zimmermann}. Set $\hat{u}_0= u_{0,\Delta t} + U$. With this $\hat{u}_0$, we introduce the following time discretization:
  \begin{align}
   \hat{u}_{k+1}-\hat{u}_k -\Delta t\, \mbox{div}_x  \big(|\nabla \hat{u}_{k+1}|^{p-2}\nabla \hat{u}_{k+1} + \vec{f}(\hat{u}_{k+1})\big)
   = \int_{|z|>0}\int_{t_k}^{t_{k+1}}\eta(\hat{u}_k;z) \widetilde{N}({\rm d}z,{\rm d}t)\,.\label{eq:discrete-p-laplace}
  \end{align}
  
  \begin{prop}
  Let $\Delta t>0$ be small and $\hat{u}_0$ is defined as above. Then, for any $k=0,1,2,\cdots, N-1$, there exists a unique $\mathcal{F}_{t_{k+1}}$-measurable $W_0^{1,p}(D)$-valued random variable
 $\hat{u}_{k+1}$ such that for any $v\in W_0^{1,p}(D)$, the following variational formula holds
  \begin{align}
    &\int_{D} \Big((\hat{u}_{k+1}-\hat{u}_k)v + \Delta t \big\{ |\nabla \hat{u}_{k+1}|^{p-2}\nabla \hat{u}_{k+1} + \vec{f}(\hat{u}_{k+1})\big\}\cdot \grad v \Big)\,{\rm d}x \notag \\
    &= \int_{D} \int_{t_k}^{t_{k+1}} \int_{|z|>0} \eta(\hat{u}_k;z)\,v\, \widetilde{N}({\rm d}z,{\rm d}s)\,{\rm d}x. \label{variational_formula_discrete}
   \end{align}
   \textit{i.e.,} $\mathbb{P}$-a.s., $\hat{u}_{k+1}$ is a unique weak solution to the problem \eqref{eq:discrete-p-laplace}.
  \end{prop}
  \begin{proof}  Let $\Delta t >0$ be a fixed small number. Define an operator $\mathcal{A}: W_0^{1,p}(D)\goto W^{-1,p^\prime}(D)$ via
   \begin{align*}
  \big\langle \mathcal{A}u, v \big\rangle := \int_{D}\Big(uv + \Delta t\big\{ |\nabla u|^{p-2}\nabla u + \vec{f}(u)\big\}\cdot \grad v \Big)\,{\rm d}x\,,
  \quad  \forall\,u,v \in W_0^{1,p}(D)\,.
  \end{align*}
  Then, $\mathcal{A}$ is a coercive pseudo-monotone operator and hence by Brezis' theorem $\mathcal{A}$ is onto $W^{-1,p^\prime}(D)$, see \cite[Theorem $2.6$]{Roubicek}. Arguing similarly as in the proof of 
  \cite[Lemma $1$]{zimmermann}, we infer that $\mathcal{A}$ is injective and $\mathcal{A}^{-1}:W^{-1,p^\prime}(D) \goto W_0^{1,p}(D)$ is 
continuous.
\vspace{.1cm}

Let $ \displaystyle X_k:= \hat{u}_k + \int_{|z|>0}\int_{t_k}^{t_{k+1}}\eta(\hat{u}_k;z) \widetilde{N}({\rm d}z,{\rm d}t)$. Then, thanks to the assumption \ref{A3} and It\^{o}-L\'{e}vy isometry, we have 
\begin{align*}
 \mathbb{E}\big[ \|X_k\|_{L^2(D)}^2\big] & \le 2 \mathbb{E}\big[\|\hat{u}_k\|_{L^2(D)}^2\big] + 2\lambda^* \Delta t\,  \mathbb{E}\Big[\int_{|z|>0} \|\hat{u}_{k}\|_{L^2(D)}^2 (1\wedge |z|^2)\,m({\rm d}z)
 \Big] \\
 & \le C(\Delta t, \lambda, c_\eta)\mathbb{E}\big[\|\hat{u}_k\|_{L^2(D)}^2\big].
\end{align*}
Therefore for a.s. $\omega \in \Omega$, $X_k \in L^2(D)$, and hence $\hat{u}_{k+1}= \mathcal{A}^{-1}X_k$. Note that $\hat{u}_{k+1}$ is $\mathcal{F}_{t_{k+1}}$-measurable 
if we assume that $\hat{u}_k$ is $\mathcal{F}_{t_{k}}$-measurable. Thus the assertion follows by induction. This completes the proof.
\end{proof}
\subsection{\textbf{\textit{A-priori} estimate}}

We choose a test function $v= \hat{u}_{k+1}$ in \eqref{variational_formula_discrete}, and use Young's inequality, \ref{A3}-\ref{A4}, It\^{o}-L\'{e}vy isometry, and the identity 
$ (a-b)a= \frac{1}{2}\big[a^2 + (a-b)^2 -b^2\big]$ for all $a,b \in \R$ to have, after taking the expectation and recalling $ \displaystyle \int_{D} \vec{f}(v)\cdot \grad v\,dx=0$ for any $v\in W_0^{1,p}(D)$, 
\begin{align*}
 &\frac{1}{2} \Big\{ \mathbb{E}\big[\|\hat{u}_{k+1}\|_{L^2(D)}^2\big]+ \mathbb{E}\big[\|\hat{u}_{k+1}-\hat{u}_k||_{L^2(D)}^2 \big]- \mathbb{E}\big[\|\hat{u}_k\|_{L^2(D)}^2\big] \Big\} +
 \Delta t\, \mathbb{E}\Big[\|\grad \hat{u}_{k+1}\|_{L^p(D)}^p\Big] \notag \\
  & \le  \frac{1}{4}  \mathbb{E}\Big[\|\hat{u}_{k+1}-\hat{u}_k\|_{L^2(D)}^2\Big] + C\,\Delta t\, \mathbb{E}\big[\|\hat{u}_k\|_{L^2(D)}^2\big]. 
\end{align*}
An application of discrete Gronwall's lemma then implies 
\begin{align}
 \sup_{0\le n\le N}\mathbb{E}\Big[\|\hat{u}_{n}\|_{L^2(D)}^2\Big] +  \sum_{k=0}^{N-1} \mathbb{E}\Big[\|\hat{u}_{k+1}-\hat{u}_k\|_{L^2(D)}^2\Big] +
  \Delta t \sum_{k=0}^{N-1} \mathbb{E}\Big[\|\grad \hat{u}_{k+1}\|_{L^p(D)}^p\Big]  \le  C\,. \label{a-prioriestimate:1}
\end{align}
Moreover, we can easily show that $ \displaystyle \mathbb{E}\Big[\sup_{0\le n\le N}\|\hat{u}_{n}\|_{L^2(D)}^2\Big] \le C$. 
\vspace{.1cm}

We would like to define certain processes defined on the whole time interval $[0,T]$ in terms of the discrete solutions $\{ \hat{u}_k\}$, and derive a-priori estimate. Like in \cite{zimmermann}, we
introduce the right-continuous step function $u_{\Delta t}(t)$, left-continuous $\{\mathcal{F}_t\}$-adapted step function $\bar{u}_{\Delta t}(t)$, square-integrable $\{\mathcal{F}_t\}$-martingale 
$ B_{\Delta t}(t)$ and the piecewise affine functions $\tilde{u}_{\Delta t}(t)$ and $\tilde{B}_{\Delta t}(t)$ as
\begin{align*}
 u_{\Delta t}(t): &= \sum_{k=0}^{N-1} \hat{u}_{k+1} {\bf 1}_{[t_k,t_{k+1})}(t)\,; \quad 
 \bar{u}_{\Delta t}(t)= \sum_{k=0}^{N-1} \hat{u}_{k} {\bf 1}_{(t_k,t_{k+1}]}(t)~~\text{with}~~
 \bar{u}_{\Delta t}(0)=\hat{u}_0\,, \\
 B_{\Delta t}(t):&= \int_0^t \int_{|z|>0} \eta(u_{\Delta t}(s);z)\widetilde{N}({\rm d}z,{\rm d}s)\,,
\end{align*}
and 
\begin{equation}\label{eq:defi-piecewise-affine-function}
 \begin{aligned}
  \tilde{u}_{\Delta t}(t) & := \sum_{k=0}^{N-1}\Big( \frac{\hat{u}_{k+1}-\hat{u}_k}{\Delta t}(t-t_k) + \hat{u}_k\Big) {\bf 1}_{[t_k,t_{k+1})}(t)~~\text{with}~~ \tilde{u}_{\Delta t}(T)= \hat{u}_N\,, \\
   \tilde{B}_{\Delta t}(t) & := \sum_{k=0}^{N-1}\Big( \frac{B_{\Delta t}(t_{k+1})-B_{\Delta t}(t_k)}{\Delta t}(t-t_k)
   + B_{\Delta t}(t_k)\Big) {\bf 1}_{[t_k,t_{k+1})}(t)\,.
 \end{aligned}
\end{equation}
A straightforward calculation shows that 
\begin{align*}
\begin{cases}
 \underset{t\in [0,T]}\sup\, \mathbb{E}\big[\|u_{\Delta t}(t)\|_{L^2(D)}^2\big] = \underset{0\le k\le N-1}\max\, \mathbb{E}\big[\|\hat{u}_{k+1}\|_{L^2(D)}^2\big]\,, \\
  \mathbb{E}\big[ \underset{t\in [0,T]}\sup\,\|u_{\Delta t}(t)\|_{L^2(D)}^2\big] \le \mathbb{E}\big[ \underset{0\le k\le N-1}\max\, \|\hat{u}_{k+1}\|_{L^2(D)}^2\big]\,, \\
 \mathbb{E}\big[\|u_{\Delta t}-\tilde{u}_{\Delta t}\|_{L^2(D_T)}^2\big] \le \displaystyle \Delta t \sum_{k=0}^{N-1} \mathbb{E}\big[\|\hat{u}_{k+1}-\hat{u}_k\|_{L^2(D)}^2\big]\,.
  \end{cases}
\end{align*}
In view of the above definitions and \textit{a-priori} estimate \eqref{a-prioriestimate:1}, we arrive at the following lemma.
\begin{lem}
 There exists a constant $C>0$, independent of $\Delta t$, such that 
 \begin{equation}\label{esti:a-priori-2}
  \begin{aligned}
  & \sup_{t\in [0,T]} \mathbb{E}\Big[\|u_{\Delta t}(t)\|_{L^2(D)}^2\Big]= \sup_{t\in [0,T]} \mathbb{E}\Big[\|\tilde{u}_{\Delta t}(t)\|_{L^2(D)}^2\Big] \le C\,, \\
   &  \mathbb{E}\Big[ \sup_{t\in [0,T]}\|u_{\Delta t}(t)\|_{L^2(D)}^2\Big]=  \mathbb{E}\Big[\sup_{t\in [0,T]}\|\tilde{u}_{\Delta t}(t)\|_{L^2(D)}^2\Big] \le C \,,\\
  & \mathbb{E}\Big[\int_0^T \int_{D} |\grad u_{\Delta t}(t)|^p\,{\rm d}x\,{\rm d}t\Big] \le C; \quad  \mathbb{E}\Big[\int_0^T \int_{D} |\grad \tilde{u}_{\Delta t}(t)|^p\,{\rm d}x\,{\rm d}t\Big] \le C \,,\\
  &\mathbb{E}\Big[\|u_{\Delta t}-\tilde{u}_{\Delta t}\|_{L^2(D_T)}^2\Big] \le C\,\Delta t\,.
  \end{aligned}
 \end{equation}
\end{lem}
 Thanks to \eqref{esti:initial-approx} and \eqref{esti:a-priori-2}, one can easily show the following estimate:
\begin{align}
 \mathbb{E}\Big[ \big\|\tilde{u}_{\Delta t}\big\|_{L^p(0,T; W_0^{1,p}(D))}^p \Big]
 &\le C \mathbb{E}\Big[ \int_0^T \|\grad u_{\Delta t}(t)\|_{L^p(D)}^p\,{\rm d}t + \Delta t\,\|\grad \hat{u}_0\|_{L^p(D)}^p\Big] \notag \\
 & \le C\,\Big(\|u_0\|_{L^2(D)}^2 + \mathbb{E}\big[\|U\|_{W^{1,p}(D)}^p\big]\Big)\,,\label{esti:tilde-lp-w1p}
\end{align}
for some constant $C>0$, independent of $\Delta t$.

\subsection{Tightness of the sequence $\mathcal{L}(\tilde{u}_{\Delta t})$} In this subsection, we will show that the laws of the sequence $\tilde{u}_{\Delta t}$, denoted by 
$\mathcal{L}(\tilde{u}_{\Delta t})$, is tight on some appropriate functional space. To do so, analogous to those considered in \cite{brzezniakmotyl2013, Metivier1988, Viot1988}, we define
\begin{align*}
 \mathcal{Z}:= \mathbb{D}([0,T]; W^{-1,p^\prime}(D))\cap \mathbb{D}([0,T]; L^2_w(D))\cap L^2_w(0,T; L^2(D)) \cap L^2(0,T; L^2(D))
\end{align*}
equipped with the topology $\mathcal{T}$, the supremum of the corresponding topologies, where the functional spaces 
$\mathbb{D}([0,T]; W^{-1,p^\prime}(D)),\,L^2_w(0,T; L^2(D))$, and $\mathbb{D}([0,T]; L^2_w(D))$ endowed with the respective topologies are defined as
\begin{itemize}
 \item[1).] $\mathbb{D}([0,T]; W^{-1,p^\prime}(D)):=$ the space of c\`{a}dl\`{a}g functions $u:[0,T]\goto W^{-1,p^\prime}(D)$ with the extended Skorokhod topology\footnote[4]{
 For the Skorokhod topology, we refer to see \cite{Billingsley,Skorokhod} and references therein.}.
 \item [2).] $L^2_w(0,T; L^2(D)):=$ the space $L^2(0,T; L^2(D))$ with the weak topology.
 \item[3).] $\mathbb{D}([0,T]; L^2_w(D)):=$ the space of all weakly c\`{a}dl\`{a}g functions $u:[0,T]\goto L^2(D)$ with the weakest topology
 such that for all 
 $h\in L^2(D)$, the mapping $\mathbb{D}([0,T]; L^2_w(D))\ni u\mapsto \int_{D}u(\cdot)h\,{\rm d}x \in \mathbb{D}([0,T]; \R)$ are continuous.
\end{itemize}
\begin{defi}\label{defi:aldous-condition}
(Aldous condition) Let $(X_n)_{n\in \mathbb{N}}$ be a sequence of c\`{a}dl\`{a}g, $\{\mathcal{F}_t\}$-adapted stochastic processes in a Banach space $\mathbb{U}$. We say that $(X_n)_{n\in \mathbb{N}}$
satisfies the Aldous condition if for every $\eps>0$ and $\gamma>0$, there is $\delta>0$ such that for every sequence $(\tau_n)_{n\in \mathbb{N}}$ of $\{\mathcal{F}_t\}$-stopping times with 
$\tau_n\le T$, one has 
\begin{align*}
 \sup_{n\in \mathbb{N}}\sup_{0<\theta\le \delta} \mathbb{P}\big\{\|X_n(\tau_n + \theta)-X_n(\tau_n)\|_{\mathbb{U}}\ge \gamma\big\} \le \eps\,.
\end{align*}
\end{defi}
The following lemma ensures the Aldous condition in a separable Banach space $\mathbb{U}$ for the sequence  $(X_n)_{n\in \mathbb{N}}$; cf.~\cite[Lemma 9]{motyl2013}.
\begin{lem}\label{lem:Aldous-condition}
 Let $(\mathbb{U},\|\cdot\|_{\mathbb{U}})$ be a separable Banach space and let $(X_n)_{n\in \mathbb{N}}$ be a sequence of $\mathbb{U}$-valued random variables. Assume that for every sequence
 $(\tau_n)$ of $\{\mathcal{F}_t\}$-stopping times with $\tau_n\le T$ and $\theta\ge 0$, the following condition holds
\begin{align}
 \mathbb{E}\Big[\|X_n(\tau_n + \theta)-X_n(\tau_n)\|^\alpha_{\mathbb{U}}\Big] \le C \theta^\zeta\,, \label{inq:Aldous-condition}
\end{align}
for some $\alpha,\, \zeta >0$ and some constant $C>0$. Then the sequence $(X_n)_{n\in \mathbb{N}}$ satisfies the Aldous condition.
\end{lem}
In view of \cite[Lemma $2.5$]{rozovskii}, \cite[Theorem $2$]{motyl2013}, see also \cite[Lemma $3.3$]{brzezniakmotyl2013} and \cite[Lemma $7$]{motyl2013}, we arrive at the following useful 
theorem regarding the criterion for the tightness in $\mathcal{Z}$. For its proof, consult \cite[Corollary $1$]{motyl2013}.
\begin{thm}\label{thm:for-tightness}
 Let $(u_{\Delta t})_{\Delta t >0}$ be a sequence of  c\`{a}dl\`{a}g, $W^{-1,p^\prime}(D)$-valued stochastic processes such that 
 \begin{itemize}
  \item[i)] there exists a constant $C_1>0$ such that 
  \begin{align*}
   \sup_{\Delta t >0} \mathbb{E}\Big[ \sup_{t\in [0,T]}\|u_{\Delta t}(t)\|_{L^2(D)}\Big] \le C_1\,,
  \end{align*}
  \item[ii)] there exists a constant $C_2>0$ such that 
  \begin{align*}
   \sup_{\Delta t >0} \mathbb{E}\Big[\int_0^T\|u_{\Delta t}(t)\|_{W_0^{1,p}(D)}^2\,{\rm d}t\Big] \le C_2\,,
  \end{align*}
  \item[iii)] $(u_{\Delta t})_{\Delta t >0}$ satisfies the Aldous condition in $W^{-1,p^\prime}(D)$.
 \end{itemize}
 Then the sequence  $\big(\mathcal{L}(u_{\Delta t})\big)_{\Delta t >0}$ is tight on $(\mathcal{Z}, \mathcal{T})$.
\end{thm}
With the help of Theorem \ref{thm:for-tightness}, we prove the tightness of the laws of the sequence $\{\tilde{u}_{\Delta t}\}$ in $(\mathcal{Z}, \mathcal{T})$. 
\begin{lem}\label{lem:tightness}
 The sequence $\big(\mathcal{L}(\tilde{u}_{\Delta t})\big)_{\Delta t >0}$ is tight on $(\mathcal{Z}, \mathcal{T})$. 
\end{lem}
\begin{proof}
 Thanks to the \textit{a-priori} estimates \eqref{esti:a-priori-2} and \eqref{esti:tilde-lp-w1p}, we see that assumptions ${\rm i)}$ and ${\rm ii)}$ of Theorem \ref{thm:for-tightness} hold for the sequence 
 $(\tilde{u}_{\Delta t})_{\Delta t >0}$. Hence it suffices to prove that the sequence $(\tilde{u}_{\Delta t})_{\Delta t >0}$ satisfies the Aldous condition in $W^{-1,p^\prime}(D)$. 
 Note that, we can rewrite \eqref{eq:discrete-p-laplace} in terms of ${u}_{\Delta t},\, \tilde{u}_{\Delta t}$, and $\tilde{B}_{\Delta t}$ as 
\begin{align}
 \tilde{u}_{\Delta t}(t)&=u_{0,\Delta t} + U + \int_0^t \mbox{div}_x\big(|\grad u_{\Delta t}(s)|^{p-2} \grad u_{\Delta t}(s) + \vec{f}(u_{\Delta t}(s))\big)\,{\rm d}s + \tilde{B}_{\Delta t}(t) \notag \\
 & := \hat{u}_0 + T_1^{\Delta t}(t) + T_2^{\Delta t}(t).\label{eq:time-cont-p-laplace}
\end{align}
First note that, since the term $\hat{u}_0$ is independent of time, clearly \eqref{inq:Aldous-condition} is satisfied for any $\alpha, \zeta$. 
In view of Lemma \ref{lem:Aldous-condition}, we need to show that $T_1^{\Delta t}(t)$ and $\,T_2^{\Delta t}(t)$ satisfy the inequality \eqref{inq:Aldous-condition} for a suitable choices of 
$\alpha, \zeta$. Let $(\tau_m)$ be a sequence of stopping times with $\tau_m \le T$, and $\theta >0$. Then, by using \eqref{esti:a-priori-2} we have
\begin{align*}
& \mathbb{E}\Big[  \big\| T_1^{\Delta t}(\tau_m + \theta)-T_1^{\Delta t}(\tau_m)\big\|_{W^{-1,p^\prime}(D)}\Big] \\
 &=\mathbb{E}\Big[ \big\|\int_{\tau_m}^{\tau_m + \theta} \mbox{div}_x\big(|\grad u_{\Delta t}(s)|^{p-2} \grad u_{\Delta t}(s) + \vec{f}(u_{\Delta t}(s))\big)\,{\rm d}s\big\|_{W^{-1,p^\prime}(D)}\Big] \\
 & \le \mathbb{E}\Big[\int_{\tau_m}^{\tau_m + \theta}\big\||\grad u_{\Delta t}(s)|^{p-2} \grad u_{\Delta t}(s) + \vec{f}(u_{\Delta t}(s))\big\|_{L^{p^\prime}(D)}\,{\rm d}s\Big] \\
 & \le C\mathbb{E}\Big[\int_{\tau_m}^{\tau_m + \theta}\Big(\big\| \grad u_{\Delta t}(s)\big\|_{L^p(D)}^{p-1} + \big\|u_{\Delta t}(s)\big\|_{L^{p^\prime}(D)}\Big)\,{\rm d}s\Big] \\
 & \le C \theta^\frac{1}{p} \mathbb{E}\Big[ \int_0^T\Big( \big\| \grad u_{\Delta t}(s)\big\|_{L^p(D)}^{p} + \big\|u_{\Delta t}(s)\big\|_{L^{p^\prime}(D)}^{p^\prime}\Big)\,{\rm d}s\Big] 
 \le C \theta^\frac{1}{2}.
\end{align*}
Thus $T_1^{\Delta t}(t)$ satisfies \eqref{inq:Aldous-condition} with $\alpha=1$ and $\zeta=\frac{1}{2}$. Again, thanks to It\^{o}-L\'{e}vy isometry, the assumptions \ref{A3}-\ref{A4}, and 
\eqref{esti:a-priori-2}, and since $W_0^{1,p}(D)\hookrightarrow L^2(D)\hookrightarrow W^{-1,p^\prime}(D)$, we see that 
\begin{align*}
 & \mathbb{E}\Big[\big\| T_2^{\Delta t}(\tau_m +\theta)-T_2^{\Delta t}(\tau_m)\big\|_{W^{-1,p^\prime}(D)}^2\Big] \\
 & \le  \mathbb{E}\Big[\big\|  \int_{\tau_m}^{\tau_m + \theta} \int_{|z|>0} \eta(\bar{u}_{\Delta t}(s);z) \widetilde{N}({\rm d}z,{\rm d}s)\big\|_{W^{-1,p^\prime}(D)}^2\Big] \\
 & \le C \mathbb{E}\Big[\big\|  \int_{\tau_m}^{\tau_m + \theta} \int_{|z|>0} \eta(\bar{u}_{\Delta t}(s);z) \widetilde{N}({\rm d}z,{\rm d}s)\big\|_{L^2(D)}^2\Big] \\
 & \le C \mathbb{E}\Big[ \int_{\tau_m}^{\tau_m + \theta} \int_{|z|>0} \big\| \eta(\bar{u}_{\Delta t}(s);z)\big\|_{L^2(D)}^2 \,m({\rm d}z)\,{\rm d}s\Big] \\
 & \le C \mathbb{E}\Big[ \int_{\tau_m}^{\tau_m + \theta} \int_{|z|>0} \big\| \bar{u}_{\Delta t}(s)\big\|_{L^2(D)}^2(1\wedge |z|^2) \,m({\rm d}z)\,{\rm d}s\Big] \\
 & \le C c_{\eta} \mathbb{E}\Big[ \int_{\tau_m}^{\tau_m + \theta} \big\| \bar{u}_{\Delta t}(s)\big\|_{L^2(D)}^2\,{\rm d}s\Big]
 \le C \theta \mathbb{E}\Big[\sup_{s\in [0,T]}\big\| u_{\Delta t}(s)\big\|_{L^2(D)}^2\Big] \le C \theta.
\end{align*}
Hence $T_2^{\Delta t}(t)$ satisfies \eqref{inq:Aldous-condition} with $\alpha=2$ and $\zeta=1$. This completes the proof.
\end{proof}
\subsection{Construction of a martingale solution} Construction of a martingale solution is based on Skorokhod Theorem \cite{Jakubowski} for a non metric space. 
Note that $\mathcal{Z}$ is a locally convex topological space and there exist a sequence of continuous functions $f_m: \mathcal{Z}\goto \R$ that 
separates the points of $\mathcal{Z}$ which generates the Borel $\sigma$-algebra; cf.~\cite[Remark $2$]{motyl2013}. 
Let $\bar{\mathbb{N}}$ denotes the set of all extended natural numbers \textit{i.e.,} $\bar{\mathbb{N}}:= \mathbb{N}\cup \{ \infty\}$.
For any measurable space $(S,\mathcal{B}(S))$, we denote by $M_{\bar{\mathbb{N}}}(S)$ the set of all $\bar{\mathbb{N}}$-valued  measures on 
$(S,\mathcal{B}(S))$ endowed with the $\sigma$-field $\mathcal{M}_{\bar{\mathbb{N}}}(S)$ generated by the projection maps 
$i_B: M_{\bar{\mathbb{N}}}(S)\ni \mu\mapsto \mu(B)\in \bar{\mathbb{N}}$ for all $B\in \mathcal{B}(S)$.
Define $N_{\Delta t}({\rm d}z,{\rm d}t)=N({\rm d}z,{\rm d}t)$ for all $\Delta t>0$. For the basic properties of the stochastic integral with respect to compensated Poisson random measure $\widetilde{N}$, we refer to see 
\cite{erika2009, Watanabe-1981} and \cite{peszat}. Since $M_{\bar{\mathbb{N}}}(\R\times [0,T])$ is a separable metric space, by \cite[Theorem 3.2]{parthasarathy}, the laws of the family 
$\{ N_{\Delta t}({\rm d}z,{\rm d}t)\}$ is tight on $M_{\bar{\mathbb{N}}}(\R\times [0,T])$.
\vspace{.1cm}

Define $U_{\Delta t}= U$ for all $\Delta t >0$ and $\mathbb{X}_{U}= \big( W^{1,p}(D),w\big)$.\footnote[5]{We denote by $(\mathbb{Y},w)$ the topological space $\mathbb{Y}$ equipped with the weak topology.}
\begin{lem}\label{lem:tightness-control}
 The set $\{ \mathcal{L}(U_{\Delta t}): \Delta t >0\}$ is tight in $\mathbb{X}_{U}$.
\end{lem}
\begin{proof}
 Note that $\displaystyle \sup_{\Delta t >0} \mathbb{E}\big[\|U_{\Delta t}\|_{W^{1,p}(D)}^p \big]=  \mathbb{E}\big[\|U\|_{W^{1,p}(D)}^p \big] < + \infty$. Now for any $R>0$, the set 
 \begin{align*}
  B_R:= \big\{ U\in W^{1,p}(D):~ \|U\|_{W^{1,p}(D)} \le R\big\}
 \end{align*}
is relatively compact in $\mathbb{X}_{U}$ and 
\begin{align*}
 \mathbb{P}\big( \|U_{\Delta t}\|_{W^{1,p}(D)} \ge R\big) \le \frac{1}{R^p} \mathbb{E}\big[\|U_{\Delta t}\|_{W^{1,p}(D)}^p \big] \le \frac{C}{R^p}\,,
\end{align*}
which yields the proof.
\end{proof}
By Lemma \ref{lem:tightness}, the set of measures $\big(\mathcal{L}(\tilde{u}_{\Delta t})\big)_{\Delta t >0}$ is
tight on $(\mathcal{Z}, \mathcal{T})$. Hence, in view of Lemma \ref{lem:tightness-control}, the set $\big\{ \mathcal{L}(\tilde{u}_{\Delta t}, U_{\Delta t}, N_{\Delta t}): \Delta t >0\big\}$ is tight on 
$\mathcal{X}:=\mathcal{Z}\times \mathbb{X}_U \times M_{\bar{\mathbb{N}}}(\R\times [0,T])$. Note that the space $\mathcal{X}$ is non-metric space, and hence our compactness argument is based on the Jakubowski-Skorokhod 
representation theorem. Moreover, by using \cite[Corollary 2]{motyl2013}, see also \cite[Theorem $D1$]{erika2014}, we arrive at the following result.
\begin{prop}\label{prop:skorokhod-representation}
There exist a subsequence of $\{\Delta t\}$, still
we denote it by same $\{\Delta t\}$, a probability space $(\bar{\Omega}, \bar{\mathcal{F}}, \bar{\mathbb{P}})$ and, on this space $\mathcal{X}$-valued 
random variables $(u_*, U_*, N_*)$ and $(u_{\Delta t}^*, U_{\Delta t}^*, N_{\Delta t}^*)$ such that 
\begin{itemize}
 \item [a).] $\mathcal{L}(u_{\Delta t}^*, U_{\Delta t}^*, N_{\Delta t}^*)=\mathcal{L}(\tilde{u}_{\Delta t}, U_{\Delta t}, N_{\Delta t})$ for all $\Delta t >0$,
 \item [b).] $(u_{\Delta t}^*, U_{\Delta t}^*, N_{\Delta t}^*)\goto (u_*,U_*, N_*)$ in $\mathcal{X}\quad \bar{\mathbb{P}}$-a.s. as $\Delta t \goto 0$,
 \item [c).] $N_{\Delta t}^*(\bar{\omega})= N_*(\bar{\omega})$ for all $\bar{\omega}\in \bar{\Omega}$.
\end{itemize}
\end{prop}
Moreover, by \cite[Theorem $1.10.4$ \&\ Addendum $1.10.5$]{wellner}, there exist a sequence of
 perfect functions $\phi_{\Delta t}:\bar{\Omega}\to\Omega$ such that
 \begin{align}
  u_{\Delta t}^*=\tilde{u}_{\Delta t}\circ\phi_{\Delta t}\,, \quad U_{\Delta t}^*= U_{\Delta t} \circ \phi_{\Delta t}\,, \quad \mathbb{P}=\bar{\mathbb{P}}\circ \phi_{\Delta t}^{-1}\,. \label{eq:perfect-function}
 \end{align}
Let $\bar{\mathbb{F}}:= \big( \bar{\mathcal{F}}_t\big)_{t\in [0,T]}$ be the filtration defined by 
\begin{align}
 \bar{\mathcal{F}}_t:= \sigma\big\{(u_{\Delta t}^*(s),\,N_{\Delta t}^*(s),\, u_*(s)): 0\le s\le t\big\}, \quad t\in [0,T]. \label{eq:filtration}
\end{align}
Note that since $N_{\Delta t}^*(\bar{\omega})= N_*(\bar{\omega})$ for all $\bar{\omega}\in \bar{\Omega}$, the filtration obtained by replacing $N_{\Delta t}^*$ by $N_*$ in \eqref{eq:filtration} 
is equal to $\bar{\mathbb{F}}$.  Moreover, $N_{\Delta t}^*,\, N_*$ are time homogeneous Poisson random measures on $\R$ over the stochastic basis
$(\bar{\Omega}, \bar{\mathcal{F}}, \bar{\mathbb{P}},\bar{\mathbb{F}})$ with intensity measure $m({\rm d}z)$; cf.~\cite[Section 9]{erika2014}.
\vspace{.1cm}

Let us define 
\begin{equation}\label{defi:discrete-variables-new-prob-space}
 \begin{aligned}
  v_k& =\hat{u}_k\circ \phi_{\Delta t}, \quad k=0,1,\cdots, N, \\
   v_{\Delta t}(t)&= \sum_{k=0}^{N-1}v_{k+1} {\bf 1}_{[t_k, t_{k+1})}(t) \quad t\in [0,T], \\
  \bar{v}_{\Delta t}(t)&= \sum_{k=0}^{N-1} v_k {\bf 1}_{(t_k, t_{k+1}]}(t) \quad t\in (0,T]\,\,\,\text{with}\,\,\, \bar{v}_{\Delta t}(0)=u_{0,\Delta t} + U_{\Delta t}^*, \\
  B_{\Delta t}^*(t)&= \int_0^t \int_{|z|>0} \eta(\bar{v}_{\Delta t}(s);z) \widetilde{N}_{\Delta t}^*({\rm d}z,{\rm d}s).
 \end{aligned}
\end{equation}
Note that, thanks to \eqref{eq:perfect-function}, \eqref{defi:discrete-variables-new-prob-space} and \eqref{eq:discrete-p-laplace}, 
  we have, for any $k=0,1,\cdots, N-1$ and $\bar{\mathbb{P}}$ a.s., 
   \begin{align}
   & v_{k+1}-v_k -\Delta t\, \mbox{div}_x  \big(|\nabla v_{k+1}|^{p-2}\nabla v_{k+1} + \vec{f}(v_{k+1})\big)
   = \int_{|z|>0}\int_{t_k}^{t_{k+1}}\eta(v_k;z) \widetilde{N}_{\Delta t}^*({\rm d}z,{\rm d}t)\,, \label{eq:discrete-p-laplace-new-prob-space} \\
   & u_{\Delta t}^*(t)= \sum_{k=0}^{N-1}\Big( \frac{v_{k+1}-v_k}{\Delta t}(t-t_k) + v_k\Big) {\bf 1}_{[t_k,t_{k+1})}(t), \quad t\in [0,T) \,\,\text{with}
  \,\,\, u_{\Delta t}^*(T)= v_N \,.\label{eq:u-delta t-star-new}
  \end{align}
Moreover the estimate \eqref{a-prioriestimate:1} remains valid for $v_k:\, k=0,1,\cdots, N$. Furthermore, thanks to \eqref{esti:a-priori-2}-\eqref{esti:tilde-lp-w1p}, and Proposition \ref{prop:skorokhod-representation},
there hold
\begin{align}
 & \bar{\mathbb{E}}\Big[ \sup_{t\in [0,T]} \|u_{\Delta t}^*(t)\|_{L^2(D)}^2\Big]= \bar{\mathbb{E}}\Big[ \sup_{t\in [0,T]} \|v_{\Delta t}(t)\|_{L^2(D)}^2\Big] 
  \le C\big( \|u_0\|_{L^2(D)}^2  + \bar{\mathbb{E}}[\|U_*\|_{L^2(D)}^2]\big)\,, \label{esti:l2-new} \\
 & \bar{\mathbb{E}}\Big[\|u_{\Delta t}^*\|_{L^p(0,T; W_0^{1,p}(D))}\Big] \le C \bar{\mathbb{E}}\Big[ \|u_0\|_{L^2(D)}^2 + \|U_*\|_{W^{1,p}(D)}^p\Big]\,. \label{esti:lp-grad-new}
\end{align}
\begin{lem}\label{lem:convergence-1}
We have the following:
\begin{itemize}
 \item [i)] $u_{\Delta t}^*\goto u_*$ in $L^q\big(\bar{\Omega}; L^2(0,T; L^2(D))\big)$ for all $1\le q<p$. \\
 \item[ii)] $v_{\Delta t} \goto u_*$ in $L^2\big(\bar{\Omega}; L^2(0,T; L^2(D))\big)$. \\
 \item [iii)] $u_{\Delta t}^* \stackrel{*}{\rightharpoonup} u_*$ in $L_w^2\big(\bar{\Omega}; L^\infty(0,T; L^2(D))\big)$.
\end{itemize}
\end{lem}
\begin{proof} We use the estimates \eqref{esti:l2-new}-\eqref{esti:lp-grad-new} to prove the lemma. 
\vspace{.2cm}

\noindent Proof of ${\rm i).}$ In view of \eqref{eq:u-delta t-star-new}, \eqref{esti:lp-grad-new} and the definition of $v_{\Delta t}$ in \eqref{defi:discrete-variables-new-prob-space}, we see
that the sequence $\{u_{\Delta t}^*\}$ is uniformly bounded in $L^p(\bar{\Omega}; L^2(0,T; L^2(D)))$ and therefore equi-integrable in $L^q(\bar{\Omega}; L^2(0,T; L^2(D)))$ for all $1\le q<p$. Since 
$\bar{\mathbb{P}}$-a.s., $u_{\Delta t}^* \goto u_*$ in $\mathcal{Z}$ (in particular, $u_{\Delta t}^* \goto u_*$ in $L^2(0,T;L^2(D))$), by Vitali convergence theorem we conclude that
${\rm i)}$ holds as well.
\vspace{.1cm}

\noindent Proof of {\rm ii)}. A straightforward calculation reveals that 
\begin{align}
 \bar{\mathbb{E}}\Big[ \int_{0}^T \|u_{\Delta t}^*(t)-v_{\Delta t}\|_{L^2(D)}^2\,{\rm d}t\Big] \le C \Delta t. \label{inq:diff-new-1}
\end{align}
Thanks to {\rm i)}, we see that $u_{\Delta t}^* \goto u_*$ in $L^2\big(\bar{\Omega}; L^2(0,T; L^2(D))\big)$ and hence ${\rm ii)}$ follows from \eqref{inq:diff-new-1}.
\vspace{.1cm}

\noindent Proof of {\rm iii)}. Note that, by \eqref{esti:l2-new}, the sequence $\{u_{\Delta t}^*\}$ is uniformly bounded in \\
$L^2(\bar{\Omega}; L^\infty(0,T;L^2(D)))$. Since $L^2(\bar{\Omega}; L^\infty(0,T;L^2(D)))$ is isomorphic to the space \\ $\big(L^2(\bar{\Omega}; L^1(0,T;L^2(D)))\big)^*$, by Banach Alaoglu theorem
there exist a subsequence, still denoted by $\{u_{\Delta t}^*\}$, and $Y\in L^2(\bar{\Omega}; L^\infty(0,T;L^2(D)))$ such that for all $\psi \in 
 L^2(\bar{\Omega}; L^1(0,T;L^2(D)))$
\begin{align*}
 \bar{\mathbb{E}}\Big[\int_0^T\int_{D} u_{\Delta t}^*(t,x)\psi(t,x)\,{\rm d}x\,{\rm d}t\Big] \goto  \bar{\mathbb{E}}\Big[\int_0^T\int_{D} Y(t,x)\psi(t,x)\,{\rm d}x\,{\rm d}t\Big]\,.
\end{align*}
Observe that, thanks to {\rm i)}, $u_{\Delta t}^*\rightharpoonup u_*$ in $L^2\big(\bar{\Omega}; L^2(0,T; L^2(D))\big)$. Since $L^2\big(\bar{\Omega}; L^2(0,T; L^2(D))\big)$ is a dense subspace of
$L^2\big(\bar{\Omega}; L^1(0,T;L^2(D))\big)$, we conclude that 
\begin{align*}
Y=u_*\quad \text{and}\quad u_*\in L^2\big(\bar{\Omega}; L^\infty(0,T;L^2(D))\big)\,.
\end{align*}
This completes the proof. 
\end{proof}

\begin{lem}\label{lem:conv-martingale-1}
For all $\phi\in W_0^{1,p}(D)$, the following holds
 \begin{align}
 \lim_{\Delta t \goto 0} \bar{\mathbb{E}}\Big[\int_0^T \Big|\int_0^t \int_{|z|>0} \big\langle \eta(\bar{v}_{\Delta t}(s);z)- \eta(u_*(s-);z),
 \phi\big\rangle \widetilde{N}_*({\rm d}z,{\rm d}s)\Big|^2\,{\rm d}t\Big]=0\,.
 \label{limit-martingale}
\end{align}
\end{lem}
\begin{proof}
By using \ref{A3}-\ref{A4}, we observe that for any $\phi\in L^2(D)$,
\begin{align*}
 & \bar{\mathbb{E}}\Big[\int_0^t \int_{|z|>0} \Big|\Big( \eta(\bar{v}_{\Delta t}(s);z)- \eta(u_*(s-);z), \phi\Big)_{L^2(D)}\Big|^2\, m({\rm d}z)\,{\rm d}s\Big] \\
 & \le \|\phi\|_{L^2(D)}^2 \bar{\mathbb{E}}\Big[ \int_0^t \int_{|z|>0} \|\eta(\bar{v}_{\Delta t}(s);z)- \eta(u_*(s-);z)\|_{L^2(D)}^2\,m({\rm d}z)\,{\rm d}s\Big] \\
 & \le C \|\phi\|_{L^2(D)}^2 \bar{\mathbb{E}}\Big[\int_0^T \|\bar{v}_{\Delta t}(s)- u_*(s-)\|_{L^2(D)}^2\,{\rm d}s\Big]. 
\end{align*}
Note that by ${\rm i)}$ of Lemma \ref{lem:convergence-1}, $u_{\Delta t}^* \goto u_*$ in $L^2(\bar{\Omega},L^2(0,T;L^2(D)))$, and the same holds for $\bar{v}_{\Delta t}$.
Hence
\begin{align}
\lim_{\Delta t \goto 0} \bar{\mathbb{E}}\Big[\int_0^t \int_{|z|>0} \Big|\Big( \eta(\bar{v}_{\Delta t}(s);z)- \eta(u_*(s-);z), \phi\Big)_{L^2(D)}\Big|^2\, m({\rm d}z)\,{\rm d}s\Big]=0.\label{eq:limit-1-martingale}
\end{align}
Moreover, by the assumptions \ref{A3}-\ref{A4} and \eqref{esti:l2-new} along with the fact that \\ $u_*\in L^2(\bar{\Omega}; L^\infty(0,T;L^2(D)))$, we have
\begin{align}
 \bar{\mathbb{E}}\Big[\int_0^t \int_{|z|>0} \Big|\Big( \eta(\bar{v}_{\Delta t}(s);z)- \eta(u_*(s-);z), \phi\Big)_{L^2(D)}\Big|^2\, m({\rm d}z)\,{\rm d}s\Big]\le C \label{esti:limit-1-martingale}
\end{align}
for some constant $C>0$. Furthermore, by using the properties of the stochastic integral with respect to the compensated Poisson random measure and the fact that 
$N_{\Delta t}^*=N_*$, we have 
\begin{align*}
 & \bar{\mathbb{E}}\Big[\Big|\int_0^t \int_{|z|>0} \Big( \eta(\bar{v}_{\Delta t}(s);z)- \eta(u_*(s-);z), \phi\Big)_{L^2(D)}\widetilde{N}_*({\rm d}z,{\rm d}s)\Big|^2\Big] \\
 & = \bar{\mathbb{E}}\Big[\int_0^t \int_{|z|>0} \Big|\Big( \eta(\bar{v}_{\Delta t}(s);z)- \eta(u_*(s-);z), \phi\Big)_{L^2(D)}\Big|^2\, m({\rm d}z)\,{\rm d}s\Big].
\end{align*}
Therefore by \eqref{eq:limit-1-martingale} and \eqref{esti:limit-1-martingale}, we have for all $\phi\in L^2(D)$
\begin{align*}
 & \lim_{\Delta t \goto 0}\bar{\mathbb{E}}\Big[\Big|\int_0^t \int_{|z|>0} \Big( \eta(\bar{v}_{\Delta t}(s);z)- \eta(u_*(s-);z), \phi\Big)_{L^2(D)}\widetilde{N}_*({\rm d}z,{\rm d}s)\Big|^2\Big]=0 \\
 &\text{and}, \quad \bar{\mathbb{E}}\Big[\Big|\int_0^t \int_{|z|>0} \Big( \eta(\bar{v}_{\Delta t}(s);z)- \eta(u_*(s-);z), \phi\Big)_{L^2(D)}\widetilde{N}_*({\rm d}z,{\rm d}s)\Big|^2\Big]\le C.
\end{align*}
Thus, one can use dominated convergence theorem to conclude 
\begin{align*}
 \lim_{\Delta t \goto 0} \bar{\mathbb{E}}\Big[\int_0^T \Big|\int_0^t \int_{|z|>0} \Big( \eta(\bar{v}_{\Delta t}(s);z)- \eta(u_*(s-);z), \phi\Big)_{L^2(D)}\widetilde{N}_*({\rm d}z,{\rm d}s)\Big|^2\,{\rm d}t\Big]=0
\end{align*}
for any $\phi \in L^2(D)$. Since $W_0^{1,p}(D)\subset L^2(D)$, \eqref{limit-martingale} holds true for all $\phi \in W_0^{1,p}(D)$. This finishes the proof.
\end{proof}
Define the piecewise affine function 
\begin{align}
 b_{\Delta t}^*(t): = \sum_{k=0}^{N-1}\Big( \frac{B_{\Delta t}^*(t_{k+1})-B_{\Delta t}^*(t_k)}{\Delta t}(t-t_k)
   + B_{\Delta t}^*(t_k)\Big) {\bf 1}_{[t_k,t_{k+1})}(t), \quad t\in [0,T]. \notag 
\end{align}
\begin{lem}\label{lem:conv-martingale-2}
 We have 
 \begin{align*}
 \bar{\mathbb{E}}\Big[\int_0^T \|B_{\Delta t}^*(t) -b_{\Delta t}^*(t)\|_{L^2(D)}^2\,{\rm d}t\Big]\le C \Delta t.
 \end{align*}
\end{lem}
\begin{proof}
 Note that for any $t\in [t_k, t_{k+1})$, we have $\bar{\mathbb{P}}$-a.s.,
 \begin{align}
  &\|B_{\Delta t}^*(t) -b_{\Delta t}^*(t)\|_{L^2(D)}^2 \notag \\
  & = \big\| \int_{t_k}^{t} \int_{|z|>0} \eta(\bar{v}_{\Delta t}(s;z))\widetilde{N}_*({\rm d}z,{\rm d}s) - \frac{t-t_k}{\Delta t} \int_{t_k}^{t_{k+1}}
  \int_{|z|>0} \eta(\bar{v}_{\Delta t}(s;z))\widetilde{N}_*({\rm d}z,{\rm d}s)\big\|_{L^2(D)}^2 \notag  \\
  & \le 2 \big\| \int_{t_k}^{t_{k+1}}\int_{|z|>0} \eta(\bar{v}_{\Delta t}(s);z)\widetilde{N}_*({\rm d}z,{\rm d}s)\big\|_{L^2(D)}^2. \label{esti:a1}
 \end{align}
 Thanks to It\^{o}-L\'{e}vy isometry, the assumption \ref{A3}, and the estimate \eqref{esti:l2-new} along with \eqref{esti:a1}, we obtain 
 \begin{align*}
   & \bar{\mathbb{E}}\Big[\int_0^T \|B_{\Delta t}^*(t) -b_{\Delta t}^*(t)\|_{L^2(D)}^2\,{\rm d}t\Big] \notag \\
 &= \sum_{k=0}^{N-1}  \bar{\mathbb{E}}\Big[\int_{t_k}^{t_{k+1}}\|B_{\Delta t}^*(t) -b_{\Delta t}^*(t)\|_{L^2(D)}^2\,{\rm d}t\Big] \\
 & \le 2 \sum_{k=0}^{N-1} \int_{t_k}^{t_{k+1}} \bar{\mathbb{E}}\Big[\| \int_{t_k}^{t_{k+1}}\int_{|z|>0} \eta(\bar{v}_{\Delta t}(s);z)\widetilde{N}_*({\rm d}z,{\rm d}s)\|_{L^2(D)}^2\Big]\,{\rm d}t \\
 & \le 2 \sum_{k=0}^{N-1} \int_{t_k}^{t_{k+1}} \bar{\mathbb{E}}\Big[ \int_{t_k}^{t_{k+1}}\int_{|z|>0} \|\bar{v}_{\Delta t}(s)\|_{L^2(D)}^2\,(1\wedge |z|^2)\,m({\rm d}z)\,{\rm d}s\Big]\,{\rm d}t \\
 & \le C \Delta t \,\bar{\mathbb{E}}\Big[ \int_{0}^{T}\|\bar{v}_{\Delta t}(s)\|_{L^2(D)}^2\,{\rm d}s\Big] \le C \Delta t.
 \end{align*}
 This completes the proof.
\end{proof}
\begin{lem} \label{lem:convergence-allterm}
 The following holds: for all $\phi\in W_0^{1,p}(D)$
 \begin{align}
 &\lim_{\Delta t\goto 0} \bar{\mathbb{E}}\Big[\big|\big(u_{\Delta t}^*(0)-u_*(0),\phi\big)_{L^2(D)}\big|\Big]=0\,,\label{eq:limit-0} \\
 & \lim_{\Delta t \goto 0} \bar{\mathbb{E}}\Big[ \int_0^T \Big|\big\langle b_{\Delta t}^*(t), \phi \big\rangle- \big\langle \int_0^t \int_{|z|>0} \eta(u_*(s-,\cdot);z) \widetilde{N}_*({\rm d}z,{\rm d}s),
  \phi \big\rangle\Big|\,{\rm d}t\Big]=0\,, \label{eq:limit-1} \\
  & \lim_{\Delta t \goto 0} \bar{\mathbb{E}}\Big[ \int_0^T \Big| \int_0^t \big\langle {\rm div}_x \big( \vec{f}(\bar{v}_{\Delta t}(s))- \vec{f}(u_*(s))\big), \phi\big\rangle\, {\rm d}s \Big|\,{\rm d}t\Big]=0\,, 
  \label{eq:limit-2}
 \end{align}
 and there exists $G\in L^{p^\prime}(\bar{\Omega}\times D_T)^d$ such that 
 \begin{align}
 \lim_{\Delta t \goto 0} \bar{\mathbb{E}}\Big[ \int_0^T \Big| \int_0^t \big\langle {\rm div}_x \big(|\grad v_{\Delta t}(s)|^{p-2}\grad v_{\Delta t}(s) - G(s)\big), 
 \phi\big\rangle\, {\rm d}s \Big|\,{\rm d}t\Big]=0\,.\label{eq:limit-3}
\end{align}
\end{lem}
\begin{proof} We prove \eqref{eq:limit-0}-\eqref{eq:limit-3} step by step.
\vspace{.1cm}

\noindent{\em Proof of \eqref{eq:limit-0}:}
Note that $\bar{\mathbb{P}}$-a.s., $u_{\Delta t}^* \goto u_*$ in $\mathbb{D}([0,T]; L^2_w(D))$ and $u_*$ is right continuous at $t=0$. Thus, for any $\phi\in W_0^{1,p}(D)$ there holds
$\bar{\mathbb{P}}$-a.s.,
$\big(u_{\Delta t}^*(0), \phi\big)_{L^2(D)}\goto \big(u_*(0), \phi\big)_{L^2(D)} $. Therefore, one can use \eqref{esti:l2-new} and Vitali theorem to conclude \eqref{eq:limit-0}.
\vspace{.1cm}

\noindent{\em Proof of \eqref{eq:limit-1}:}
Notice that, for any $\phi\in W_0^{1,p}(D)$
\begin{align*}
& \big\langle b_{\Delta t}^*(t), \phi \big\rangle- \big\langle \int_0^t \int_{|z|>0} \eta(u_*(s-,\cdot);z) \widetilde{N}_*({\rm d}z,{\rm d}s), \phi \big\rangle \\
&=\Big\langle b_{\Delta t}^*(t)- B_{\Delta t}^*(t) + B_{\Delta t}^*(t)-\int_0^t \int_{|z|>0} \eta(u_*(s-,\cdot);z) \widetilde{N}_*({\rm d}z,{\rm d}s), \phi \Big\rangle \\
& \le \|\phi\|_{W_0^{1,p}(D)}\| b_{\Delta t}^*(t)- B_{\Delta t}^*(t)\|_{W^{-1,p}(D)} + \Big| \Big\langle B_{\Delta t}^*(t)-\int_0^t \int_{|z|>0} \eta(u_*(s-,\cdot);z)
\widetilde{N}_*({\rm d}z,{\rm d}s), \phi \Big\rangle \Big| \\
& \le C\|\phi\|_{W_0^{1,p}(D)}\| b_{\Delta t}^*(t)- B_{\Delta t}^*(t)\|_{L^2(D)} \\
& \hspace{2cm}+ \Big| \int_0^t \int_{|z|>0} \big\langle \eta(\bar{v}_{\Delta t}(s,\cdot);z)- \eta(u_*(s-,\cdot);z), \phi \big\rangle
\widetilde{N}_*({\rm d}z,{\rm d}s) \Big|.
\end{align*}
One can use Lemmas \ref{lem:conv-martingale-1} and \ref{lem:conv-martingale-2} to arrive at \eqref{eq:limit-1}.
\vspace{.1cm}

\noindent{\em Proof of \eqref{eq:limit-2}:} Since $\vec{f}$ is Lipschitz continuous, we have
\begin{align*}
 \bar{\mathbb{E}}\Big[ \int_0^T \Big| \int_0^t \big\langle {\rm div}_x \big( \vec{f}(\bar{v}_{\Delta t}(s))- \vec{f}(u_*(s))\big), \phi\big\rangle\, {\rm d}s \Big|\,{\rm d}t\Big] 
 \le C \|\phi\|_{W_0^{1,p}(D)} \| \bar{v}_{\Delta t}- u_*\|_{L^2(\bar{\Omega}\times D_T)}\,,
\end{align*}
and hence \eqref{eq:limit-2} holds by recalling that $\bar{v}_{\Delta t}\rightarrow u_*$ in $L^2(\bar{\Omega}\times D_T)$.
\vspace{.1cm}

\noindent{\em Proof of \eqref{eq:limit-3}:} Thanks to $\rm ii)$ of Lemma \ref{lem:convergence-1} and 
the estimate \eqref{esti:lp-grad-new}, there exists a not relabeled subsequence of $\big\{v_{\Delta t}\big\}$ such that 
$\grad v_{\Delta t}\rightharpoonup \grad u_*$ in $L^p(\bar{\Omega}\times D_T)^d$ for $\Delta t \goto 0$.  
Moreover, since $\big| |\grad v_{\Delta t}|^{p-2}\grad v_{\Delta t}\big|^{p^\prime}=|\grad v_{\Delta t}|^p$, there exists
$G\in L^{p^\prime}(\bar{\Omega}\times D_T)^d$ such that $
 |\grad v_{\Delta t}|^{p-2}\grad v_{\Delta t} \rightharpoonup G $ in $L^{p^\prime}(\bar{\Omega}\times D_T)^d $
for the same subsequence and $\Delta t \goto 0$. Thus, it is easy to conclude that for any $\phi\in W_0^{1,p}(D)$
\begin{align*}
 \lim_{\Delta t \goto 0} \bar{\mathbb{E}}\Big[ \int_0^T \Big| \int_0^t \big\langle {\rm div}_x \big(|\grad v_{\Delta t}|^{p-2}\grad v_{\Delta t}(s)
 - G(s)\big), \phi\big\rangle\, {\rm d}s \Big|\,{\rm d}t\Big]=0\,,
\end{align*}
i.e., \eqref{eq:limit-3} holds true. 
This completes the proof. 
\end{proof}
\subsection{Proof of Theorem \ref{thm:existence-weak}}\label{subsec:existence}
In this subsection, we use Lemmas \ref{lem:convergence-1} and \ref{lem:convergence-allterm} to prove existence
of a weak solution of \eqref{eq:p-laplace} in the sense of Definition \ref{defi:weak-solun} in three steps. Moreover, we show path-wise uniqueness of weak solutions
of the problem \eqref{eq:p-laplace} with respect to the same stochastic basis and a given control.
\vspace{.1cm}

\noindent{\bf Step ${\rm i)}$:} We define the functionals for all $\phi\in W_0^{1,p}(D)$,
\begin{equation*}
\begin{aligned}
 \mathcal{K}_{\Delta t}(\tilde{u}_{\Delta t}, U, \widetilde{N}; \phi) &= \big( \hat{u}_0, \phi\big)_{L^2(D)} + 
 \int_0^t \big\langle \mbox{div}_x \big(|\grad u_{\Delta t}|^{p-2}\grad u_{\Delta t}(s)\big) , \phi\big\rangle \,{\rm d}s \\
 & \hspace{2cm} + \int_0^t \big\langle \mbox{div}_x \vec{f}(u_{\Delta t}(s)), \phi\big\rangle \,{\rm d}s
 + \big\langle \tilde{B}_{\Delta t}(t), \phi \big\rangle\,, \\
 \mathcal{K}_{\Delta t}^*(u_{\Delta t}^*, U_{\Delta t}^*, \widetilde{N}_{\Delta t}^*; \phi) &= \big( u_{\Delta t}^*(0), \phi\big)_{L^2(D)} + 
 \int_0^t \big\langle \mbox{div}_x \big(|\grad v_{\Delta t}|^{p-2}\grad v_{\Delta t}(s)\big) , \phi\big\rangle \,{\rm d}s \\ 
 & \hspace{2cm}+ \int_0^t \big\langle \mbox{div}_x\vec{f}(v_{\Delta t}(s)),  \phi\big\rangle \,{\rm d}s +  \big\langle b_{\Delta t}^*(t), \phi \big\rangle\,, \\
 \mathcal{K}_*(u_*, U_*, \widetilde{N}_*; \phi) &= \big( u_*(0), \phi\big)_{L^2(D)} + \int_0^t \big\langle \mbox{div}_x \big(G(s) + \vec{f}(u_*(s))\big), \phi\big\rangle \,{\rm d}s \\
 & \hspace{2cm} + \Big\langle \int_0^t \int_{|z|>0} \eta(u_*(s-,\cdot);z) \widetilde{N}_*({\rm d}z,{\rm d}s), \phi \Big\rangle\,.
\end{aligned}
\end{equation*}
In view of Lemma \ref{lem:convergence-allterm}, we conclude that 
\begin{align}
 \lim_{\Delta t \goto 0} \big\|\mathcal{K}_{\Delta t}^*(u_{\Delta t}^*, U_{\Delta t}^*, \widetilde{N}_{\Delta t}^*; \phi)- \mathcal{K}_*(u_*,U_*, \widetilde{N}_*; \phi)
 \big\|_{L^1(\bar{\Omega}\times[0,T])}=0. \label{eq:final-0-martingale-solun}
\end{align}
Thanks to the definition of $\mathcal{K}_{\Delta t}(\tilde{u}_{\Delta t}, U, \widetilde{N}; \phi)$ and the equality \eqref{eq:time-cont-p-laplace},
we have: $\bar{\mathbb{P}}$-a.s., \\
 $ \big( \tilde{u}_{\Delta t}(t), \phi \big)_{L^2(D)}= \mathcal{K}_{\Delta t}(\tilde{u}_{\Delta t}, U, \widetilde{N}; \phi) $ for all  $t\in [0,T]$.
More precisely,
\begin{align*}
 \bar{\mathbb{E}}\Big[ \int_0^T \Big| \big( \tilde{u}_{\Delta t}(t), \phi \big)_{L^2(D)}- \mathcal{K}_{\Delta t}(\tilde{u}_{\Delta t}, \widetilde{N}; \phi)\Big|\,{\rm d}t\Big]=0\,.
\end{align*}
Since $\mathcal{L}(u_{\Delta t}^*,U_{\Delta t}^*, N_{\Delta t}^*)=\mathcal{L}(\tilde{u}_{\Delta t},U_{\Delta t}, N_{\Delta t})$ with $N_{\Delta t}({\rm d}z,{\rm d}t)=N({\rm d}z,{\rm d}t)$ and $U_{\Delta t}=U$ for all 
$\Delta t >0$, we directly have 
\begin{align}
 \int_0^T \bar{\mathbb{E}} \Big[ \Big|\big( u_{\Delta t}^*(t), \phi \big)_{L^2(D)} - \mathcal{K}_{\Delta t}^*(u_{\Delta t}^*, U_{\Delta t}^*, \widetilde{N}_{\Delta t}^*; \phi)\Big|
 \Big]\,{\rm d}t=0\,.\label{eq:final-1-martingale-solun}
\end{align}
Again, thanks to $\rm i)$ of Lemma \ref{lem:convergence-1}, we see that 
\begin{align}
 \lim_{\Delta t \goto 0} \big\| \big(u_{\Delta t}^*(\cdot), \phi \big)_{L^2(D)} - \big( u_*(\cdot), \phi \big)_{L^2(D)}\big\|_{L^1(\bar{\Omega}\times 
 [0,T])}=0\,. \label{eq:final-2-martingale-solun}
\end{align}
We combine \eqref{eq:final-0-martingale-solun}-\eqref{eq:final-2-martingale-solun} to conclude that $\bar{\mathbb{P}}$-a.s.,
for a.e. $t\in [0,T]$, and $\phi\in W_0^{1,p}(D)$
\begin{align*}
 \big( u_*(t), \phi \big)_{L^2(D)}= \big( u_*(0), \phi \big)_{L^2(D)} +  \int_0^t \big\langle \mbox{div}_x \big(G(s) + \vec{f}(u_*(s))\big), \phi\big\rangle \,{\rm d}s \notag  \\
  + \Big\langle \int_0^t \int_{|z|>0} \eta(u_*(s-,\cdot);z) \widetilde{N}_*({\rm d}z,{\rm d}s), \phi \Big\rangle\,.
\end{align*}
Note that, $u_{\Delta t}^*(0)= u_{0,\Delta t} + U_{\Delta t}^*$ and $\bar{\mathbb{P}}$-a.s., $U_{\Delta t}^* \goto U_*$ in $\mathbb{X}_U$. Since $u_{0,\Delta t}\goto u_0$ in $L^2(D)$, by using
$\rm i)$ of Lemma \ref{lem:convergence-1}, we infer that $u_*(0)= u_0 + U_*$. Hence, we obtain
\begin{align}
 \big( u_*(t), \phi \big)_{L^2(D)}= \big( u_0 + U_*, \phi \big)_{L^2(D)} +  \int_0^t \big\langle \mbox{div}_x \big(G(s) + \vec{f}(u_*(s))\big), \phi\big\rangle \,{\rm d}s \notag  \\
  + \Big\langle \int_0^t \int_{|z|>0} \eta(u_*(s-,\cdot);z) \widetilde{N}_*({\rm d}z,{\rm d}s), \phi \Big\rangle\,. \label{eq:p-laplace-new}
\end{align}
Since $u_* \in \mathbb{D}([0,T]; L_w^2(D))$, \eqref{eq:p-laplace-new} holds true for all $t\in [0,T]$ and all $\phi \in W_0^{1,p}(D)$.
\vspace{.1cm}

\noindent{\bf Step ${\rm ii)}$:} We wish to identify the function $G \in L^{p^\prime}(\bar{\Omega}\times D_T)^d$. We take the $L^2$-scalar product with $v_{k+1}$ in \eqref{eq:discrete-p-laplace-new-prob-space}
and use the identity $(a-b)a= \frac{1}{2}\big( |a|^2-|b|^2 + |a-b|^2\big) \quad \forall\, a, b\in \R$ to have 
\begin{align}
& \frac{1}{2} \bar{\mathbb{E}}\Big[ \|v_{k+1}\|_{L^2(D)}^2 - \|v_{k}\|_{L^2(D)}^2 + \|v_{k+1}-v_k\|_{L^2(D)}^2\Big]
 + \Delta t\, \bar{\mathbb{E}}\Big[ \int_{D} |\grad v_{k+1}|^{p-2}\grad v_{k+1}\cdot \grad v_{k+1}\,{\rm d}x\Big]\notag  \\
 &  \qquad - \bar{\mathbb{E}}\Big[ \Big( \int_{t_k}^{t_{k+1}}\int_{|z|>0} \eta(v_k;z)\widetilde{N}_{\Delta t}^*({\rm d}z,{\rm d}t),
 v_{k+1}-v_k\Big)_{L^2(D)}\Big] =0\,. \label{eq:discrete-1-new}
\end{align}
Since 
\begin{align*}
 & - \bar{\mathbb{E}}\Big[ \Big( \int_{t_k}^{t_{k+1}}\int_{|z|>0} \eta(v_k;z)\widetilde{N}_{\Delta t}^*({\rm d}z,{\rm d}t), v_{k+1}-v_k\Big)_{L^2(D)}\Big] \\
 & = -\frac{1}{2}\bar{\mathbb{E}}\Big[ \|v_{k+1}-v_k\|_{L^2(D)}^2\Big]  - \frac{1}{2} \bar{\mathbb{E}}\Big[ \|
 \int_{t_k}^{t_{k+1}}\int_{|z|>0} \eta(v_k;z)\widetilde{N}_{\Delta t}^*({\rm d}z,{\rm d}t)\|_{L^2(D)}^2\Big]   \\
 & \qquad + \frac{1}{2}\bar{\mathbb{E}}\Big[ \big\|\int_{t_k}^{t_{k+1}}\int_{|z|>0} \eta(v_k;z)\widetilde{N}_{\Delta t}^*({\rm d}z,{\rm d}t)-(v_{K+1}-v_k)\big\|_{L^2(D)}^2\Big],
\end{align*} 
by summing over $k=0,1,\cdots, N-1$ in \eqref{eq:discrete-1-new} and using the fact that $v_N=u_{\Delta t}^*(T)$, we get
\begin{align}
 &\frac{1}{2}\bar{\mathbb{E}}\Big[ \|u_{\Delta t}^*(T)\|_{L^2(D)}^2\Big] + \bar{\mathbb{E}}\Big[ \int_{D_T}
 |\grad v_{\Delta t}(t)|^{p-2}\grad v_{\Delta t}(t)\cdot \grad v_{\Delta t}(t)\,{\rm d}x\,{\rm d}t\Big] \notag \\
 & -\frac{1}{2}\sum_{k=0}^{N-1} \bar{\mathbb{E}}\Big[ \|\int_{t_k}^{t_{k+1}}\int_{|z|>0} \eta(v_k;z)\widetilde{N}_{\Delta t}^*({\rm d}z,{\rm d}t)\|_{L^2(D)}^2\Big] 
 \le \frac{1}{2}\bar{ \mathbb{E}}\Big[\|u_{0,\Delta t} + U_{\Delta t}^*\|_{L^2(D)}^2\Big].\label{esti:1-new}
\end{align}
Thanks to It\^{o}-L\'{e}vy isometry, we see that 
\begin{align}
 &\sum_{k=0}^{N-1} \bar{\mathbb{E}}\Big[ \|\int_{t_k}^{t_{k+1}}\int_{|z|>0} \eta(v_k;z)\widetilde{N}_{\Delta t}^*({\rm d}z,{\rm d}t)\|_{L^2(D)}^2\Big] \notag \\
 &= \bar{\mathbb{E}}\Big[ \int_0^T \int_{|z|>0} \|\eta(\bar{v}_{\Delta t}(t);z)\|_{L^2(D)}^2\,m({\rm d}z)\,{\rm d}t\Big].\label{esti:2-new}
\end{align}
Again, an application of It\^{o}-L\'{e}vy formula \cite[similar to Theorem $3.4$]{tudor} to the functional $\|u_*(t)\|_{2}^2$ in \eqref{eq:p-laplace-new} yields
\begin{align}
 &\frac{1}{2}\bar{\mathbb{E}}\Big[ \|u_*(T)\|_{L^2(D)}^2\Big] + \bar{\mathbb{E}}\Big[ \int_{D_T}
 G \cdot \grad u_* \,{\rm d}x\,{\rm d}t\Big]
  -\frac{1}{2} \bar{\mathbb{E}}\Big[ \int_0^T\int_{|z|>0} \|\eta(u_*(s-);z)\|_{L^2(D)}^2\, m({\rm d}z)\,{\rm d}s\Big] \notag \\
  & = \frac{1}{2} \bar{\mathbb{E}}\Big[\|u_0 + U_*\|_{L^2(D)}^2\Big].\label{esti:3-new}
\end{align}
Combining \eqref{esti:1-new}, \eqref{esti:2-new}, and \eqref{esti:3-new} we obtain 
\begin{align*}
 & \frac{1}{2}\bar{\mathbb{E}}\Big[ \|u_{\Delta t}^*(T)\|_{L^2(D)}^2- \|u_*(T)\|_{L^2(D)}^2\Big] 
+ \bar{\mathbb{E}}\Big[ \int_{D_T}|\grad v_{\Delta t}(t)|^{p-2}\grad v_{\Delta t}(t)\cdot \grad v_{\Delta t}(t)\,{\rm d}x\,{\rm d}t\Big] \notag \\
& \quad -\frac{1}{2}\bar{\mathbb{E}}\Big[ \int_0^T \int_{|z|>0} \Big(\|\eta(\bar{v}_{\Delta t}(s);z)\|_{L^2(D)}^2- \|\eta(u_*(s-);z)\|_{L^2(D)}^2\Big)\, m({\rm d}z)\,{\rm d}s\Big] \notag \\
& \le  \bar{\mathbb{E}}\Big[ \int_{D_T}G \cdot \grad u_* \,{\rm d}x\,{\rm d}t\Big] + \frac{1}{2} \Big\{ \bar{ \mathbb{E}}\Big[\|u_{0,\Delta t} + U_{\Delta t}^*\|_{L^2(D)}^2\Big]-
\bar{\mathbb{E}}\Big[\|u_0 + U_*\|_{L^2(D)}^2\Big]\Big\}\,.
\end{align*}
Note that 
\begin{align*}
\begin{cases}
 \displaystyle \liminf_{\Delta t >0} \bar{\mathbb{E}}\Big[ \|u_{\Delta t}^*(T)\|_{L^2(D)}^2- \|u_*(T)\|_{L^2(D)}^2\Big] \ge 0\, , \\
 \bar{ \mathbb{E}}\Big[\|u_{0,\Delta t} + U_{\Delta t}^*\|_{L^2(D)}^2\Big] \goto  \bar{ \mathbb{E}}\Big[\|u_{0} + U_*\|_{L^2(D)}^2\Big]\,,
 \end{cases}
\end{align*}
and thanks to ${\rm ii)}$ of Lemma \ref{lem:convergence-1} along with the assumptions \ref{A3} and \ref{A4}, it follows that 
\begin{align*}
 \bar{\mathbb{E}}\Big[\int_0^T \int_{|z|>0}\|\eta(\bar{v}_{\Delta t}(t);z)\|_{L^2(D)}^2\,m({\rm d}z)\,{\rm d}t\Big] \goto 
 \bar{\mathbb{E}}\Big[\int_0^T \int_{|z|>0}\|\eta(u_*(t-);z)\|_{L^2(D)}^2\,m({\rm d}z)\,{\rm d}t\Big]\,.
\end{align*}
Thus, one arrives at the following inequality 
\begin{align}
 \limsup_{\Delta t >0} \bar{\mathbb{E}}\Big[\int_{D_T}|\grad v_{\Delta t}(t)|^{p-2}\grad v_{\Delta t}(t)\cdot \grad v_{\Delta t}(t)\,{\rm d}x\,{\rm d}t\Big] 
 \le  \bar{\mathbb{E}}\Big[ \int_{D_T} G \cdot \grad u_* \,{\rm d}x\,{\rm d}t\Big]\,. \label{esti:4-new}
\end{align}
We recall that $\grad v_{\Delta t}\rightharpoonup \grad u_*$ in $ L^p(\bar{\Omega}\times D_T)^d$ and $ |\grad v_{\Delta t}|^{p-2}\grad v_{\Delta t} \rightharpoonup G$
in $ L^{p^\prime}(\bar{\Omega}\times D_T)^d$. Since $p>2$, there exists a constant $C>0$, independent of $\Delta t$, such that 
\begin{align}
 & C \limsup_{\Delta t \goto 0} \bar{\mathbb{E}}\Big[\int_{D_T} \big| \grad v_{\Delta t}- \grad u_*\big|^p\,{\rm d}x\,{\rm d}t\Big] \notag \\
 & \le \limsup_{\Delta t \goto 0}\bar{\mathbb{E}}\Big[\int_{D_T} \Big( |\grad v_{\Delta t}|^{p-2}\grad v_{\Delta t} -
 |\grad u_*|^{p-2}\grad u_* \Big)\cdot \grad(v_{\Delta t}-u_*)\,{\rm d}x\,{\rm d}t\Big] \notag \\
 & \le \limsup_{\Delta t >0} \bar{\mathbb{E}}\Big[\int_{D_T}|\grad v_{\Delta t}(t)|^{p-2}\grad v_{\Delta t}(t)\cdot \grad v_{\Delta t}(t)\,{\rm d}x\,{\rm d}t\Big]
 - \bar{\mathbb{E}}\Big[ \int_{D_T} G \cdot \grad u_* \,{\rm d}x\,{\rm d}t\Big] \le 0\,, \notag 
\end{align}
where the last inequality follows from \eqref{esti:4-new}. Therefore, since $ \grad v_{\Delta t}\rightharpoonup \grad u_*$ in $L^p(\bar{\Omega}\times D_T)^d$
we conclude that $ \grad v_{\Delta t}\goto \grad u_*$ in $L^p(\bar{\Omega}\times D_T)^d$, and hence $
  |\grad v_{\Delta t}|^{p-2}\grad v_{\Delta t} \goto  |\grad u_*|^{p-2}\grad u_*$ in $L^{p^\prime}(\bar{\Omega}\times D_T)^d$.
In other words, $G=|\grad u_*|^{p-2}\grad u_*$.
\vspace{.1cm}

\noindent{\bf Step ${\rm iii)}$:} With the identification of $G$, it follows from \eqref{eq:p-laplace-new} that 
the system \\$\bar{\pi}:=\big( \bar{\Omega}, \bar{\mathcal{F}},  \bar{\mathbb{P}}, \bar{\mathbb{F}}, N_*, u_*, U_*\big)$
is a weak solution of the problem \eqref{eq:p-laplace}. Moreover, since $$\mathcal{L}(U_{\Delta t}^*)= \mathcal{L}(U_{\Delta t})~\text{on}~ W^{1,p}(D)~\text{with}~ U_{\Delta t}=U,~\text{and}~~
\mathbb{P}\text{-a.s.,}~ U_{\Delta t}^* \goto U_*~\text{in}~\mathbb{X}_{U},$$
we see that ${\rm i)}$ in Theorem \ref{thm:existence-weak} holds. Furthermore, one can use Proposition \ref{prop:skorokhod-representation} and the 
 estimates \eqref{esti:l2-new}-\eqref{esti:lp-grad-new} to arrive at ${\rm ii)}$, Theorem \ref{thm:existence-weak}. This completes the existence proof. 
 \subsubsection{\bf On path-wise uniqueness of weak solutions:} \label{subsec:uniqueness}
 Let $(\Omega,  \mathcal{F},\mathbb{P}, \mathbb{F}, N, u_1,U)$ and \\ $(\Omega, \mathcal{F}, \mathbb{P}, \mathbb{F}, N, u_2,U)$ 
be two weak solutions of \eqref{eq:p-laplace} with a given control $U$. Let us introduce the convex approximation of the absolute value function. 
Let $\beta:\R \rightarrow \R$ be a $C^\infty$ function satisfying 
\begin{align*}
      \beta(0) = 0\,,\quad \beta(-r)= \beta(r)\,,\quad 
      \beta^\prime(-r) = -\beta^\prime(r)\,,\quad \beta^{\prime\prime} \ge 0\,,
\end{align*} 
and 
\begin{align*}
\beta^\prime(r)=
	\begin{cases}
	-1\quad \text{when} ~ r\le -1\,,\\
         \in [-1,1] \quad\text{when}~ |r|<1\,,\\
         +1 \quad \text{when} ~ r\ge 1\,.
         \end{cases}
\end{align*} 
For any $\vartheta > 0$, define $\beta_\vartheta:\R \rightarrow \R$ by 
$\beta_\vartheta(r) = \vartheta \beta(\frac{r}{\vartheta})$. 
Then
\begin{align}\label{eq:approx to abosx}
 |r|-M_1\vartheta \le \beta_\vartheta(r) \le |r|\quad 
 \text{and} \quad |\beta_\vartheta^{\prime\prime}(r)| 
 \le \frac{M_2}{\vartheta} {\bf 1}_{|r|\le \vartheta}\,,
\end{align} 
where $M_1 = \sup_{|r|\le 1}\big | |r|-\beta(r)\big |$ and 
$M_2 = \sup_{|r|\le 1}|\beta^{\prime\prime} (r)|$.
\vspace{.1cm}

We apply It\^{o}-L\'{e}vy formula to the functional $\int_{D}\beta_\vartheta(u_1(t)-u_2(t))\,{\rm d}x$ and have 
\begin{align}
 & \int_{D}\beta_\vartheta(u_1(t)-u_2(t))\,{\rm d}x \notag \\
 & = - \int_{0}^t \int_{D} \big(|\grad u_1|^{p-2}\grad u_1 - |\grad u_2|^{p-2}\grad u_2\big)
 \cdot \grad(u_1-u_2)(s)\beta_{\vartheta}^{\prime\prime}(u_1-u_2) \,{\rm d}x\,{\rm d}s \notag \\
 & - \int_{0}^t \int_{D}  \big(\vec{f}(u_1(s,x))-\vec{f}(u_2(s,x))\big)
 \cdot \grad(u_1(s,x)-u_2(s,x))\beta_{\vartheta}^{\prime\prime}(u_1-u_2) \,{\rm d}x\,{\rm d}s \notag \\
 & + \int_{0}^{t}\int_{|z|>0}\int_{D}\int_0^1 \beta_\vartheta^{\prime}\Big( (u_1 -u_2)(s-,x) + \lambda 
 \big( \eta(u_1(s-,x);z)-\eta(u_2(s-,x);z)\big)\Big) \notag \\
  & \hspace{3cm }\times \big( \eta(u_1(s-,x);z)-\eta(u_2(s-,x);z)\big)
 \,{\rm d}\lambda\,{\rm d}x\,\widetilde{N}({\rm d}z,{\rm d}s) \notag \\
 & + \int_{0}^{t}\int_{|z|>0}\int_{D}\int_0^1 (1-\lambda)  \beta_\vartheta^{\prime\prime}\Big( (u_1-u_2)(s-,x) + \lambda 
 \big( \eta(u_1(s-,x);z)-\eta(u_2(s-,x);z)\big)\Big) \notag \\
 & \hspace{3cm} \times \big( \eta(u_1(s-,x);z)-\eta(u_2(s-,x);z)\big)^2\,{\rm d}\lambda\,{\rm d}x\,m({\rm d}z)\,{\rm d}s\,. \notag 
\end{align}
Since $p>2$ and $\beta_{\vartheta}^{\prime \prime}\ge 0$, we see that 
\begin{align*}
 & - \big(|\grad u_1|^{p-2}\grad u_1 - |\grad u_2|^{p-2}\grad u_2\big)
 \cdot \grad(u_1-u_2)\beta_{\vartheta}^{\prime\prime}(u_1-u_2) \\
 & \le - C |\grad(u_1 -u_2)|^p\beta_{\vartheta}^{\prime\prime}(u_1-u_2)  \le 0\,,
\end{align*}
and therefore, we obtain 
\begin{align}
 & \mathbb{E}\Big[ \int_{D}\beta_\vartheta(u_1(t)-u_2(t))\,{\rm d}x \Big] \notag \\
 & \le \mathbb{E}\Big[ - \int_{0}^t \int_{D}  \big(\vec{f}(u_1(s,x))-\vec{f}(u_2(s,x))\big)
 \cdot \grad(u_1(s,x)-u_2(s,x))\beta_{\vartheta}^{\prime\prime}(u_1-u_2) \,{\rm d}x\,{\rm d}s\Big] \notag \\
 & ~+ \mathbb{E}\Big[ \int_{0}^{t}\int_{|z|>0}\int_{D}\int_0^1 (1-\lambda)  \beta_\vartheta^{\prime\prime}\Big( u_1(s,x) -u_2(s,x) + \lambda 
 \big( \eta(u_1(s,x);z)-\eta(u_2(s,x);z)\big)\Big) \notag \\
 & \hspace{4cm} \times \big( \eta(u_1(s,x);z)-\eta(u_2(s,x);z)\big)^2\,{\rm d}\lambda\,{\rm d}x\,m({\rm d}z)\,{\rm d}s\Big] \notag \\
 & \equiv \mathcal{A} + \mathcal{B}\,. \label{inq:sum-a-b-pathwise-uniqueness}
\end{align}
Since $\beta_{\vartheta}^{\prime \prime}(r)\le \frac{M_2}{\vartheta} \textbf{1}_{\{|r| \le \vartheta\}}$ and $\vec{f}$ is a Lipschitz continuous 
function, we have $\mathbb{P}$-a.s., 
\begin{align*}
& \big(\vec{f}(u_1)-\vec{f}(u_2)\big)\cdot \grad(u_1(s,x)-u_2(s,x))\beta_{\vartheta}^{\prime\prime}(u_1-u_2) \\
& \le c_f |u_1 -u_2|\,|\grad(u_1-u_2)| \frac{M_2}{\vartheta} \textbf{1}_{\{|u_1-u_2| \le \vartheta\}} \goto 0 \quad (\vartheta \goto 0)
\end{align*}
for almost every $(t,x)\in D_T$. Moreover
$$\big|\vec{f}(u_1)-\vec{f}(u_2)\big||\grad(u_1(s,x)-u_2(s,x)|\beta_{\vartheta}^{\prime\prime}(u_1-u_2) \le M_2 |\grad(u_1-u_2)|\in L^1(\Omega \times D_T).$$  Thus, 
by dominated convergence theorem we conclude that $\mathcal{A}\goto 0$ as $\vartheta \goto 0$.
\vspace{.1cm}

Next we move on to estimate $\mathcal{B}$. Let 
$$a= u_1(s,x)-u_2(s,x) \quad \text{and}\quad  b= \eta(u_1(s,x);z)-\eta(u_2(s,x);z).$$
Then, we have, in view of the assumption \ref{A3},
\begin{align}
 \mathcal{B} &= \mathbb{E}\Big[\int_{0}^{t}\int_{|z|>0}\int_{D}\int_0^1 (1-\lambda) b^2 \beta_\vartheta^{\prime\prime}\big( a + \lambda b\big)
\,{\rm d}\lambda\,{\rm d}x\,m({\rm d}z)\,{\rm d}s\Big] \notag \\
& \le \mathbb{E}\Big[\int_{0}^{t}\int_{|z|>0}\int_{D}\int_0^1 (1-\lambda) a^2 \beta_\vartheta^{\prime\prime}\big( a + \lambda b\big)(1\wedge |z|^2)
\,{\rm d}\lambda\,{\rm d}x\,m({\rm d}z)\,{\rm d}s\Big].\label{esti:b-0}
\end{align}
Note that $\beta_{\vartheta}^{\prime\prime}$ is non-negative and symmetric around zero. Thus we may assume, without loss of generality, that $a\ge 0$. 
Then by the assumption \ref{A3}
\begin{align}
 u_1(s,x)-u_2(s,x) + \lambda b \ge (1-\lambda^*) \big( u_1(s,x)-u_2(s,x)\big) \notag 
\end{align}
for $\lambda\in [0,1]$. In other words
\begin{align}
 0\le a\le (1-\lambda^*)^{-1}(a+ \lambda b). \label{esti:b-1}
\end{align}
We combine \eqref{esti:b-0} and \eqref{esti:b-1} to obtain 
\begin{align}
 \mathcal{B} \le C(\lambda^*) \mathbb{E}\Big[\int_{0}^{t}\int_{|z|>0}\int_{D}\int_0^1 (1-\lambda) (a+\lambda b)^2
 \beta_\vartheta^{\prime\prime}\big( a + \lambda b\big)(1\wedge |z|^2)
\,{\rm d}\lambda\,{\rm d}x\,m({\rm d}z)\,{\rm d}s\Big].\notag 
\end{align}
In view of \eqref{eq:approx to abosx}, and the assumption on $\eta$ that $\eta(0,z)=0$ for all $z\in \R$, we see that for each $\lambda\in [0,1]$
\begin{align}
 (a+\lambda b)^2  \beta_\vartheta^{\prime\prime}\big( a + \lambda b\big) \le |a+\lambda b|\textbf{1}_{\{0< |a+\lambda b|< \vartheta \}} 
  \le |a+\lambda b| \in L^1(\Omega \times D_T) \notag 
\end{align}
for $m({\rm d}z)$-almost every $z\in \R$. Again $|a+\lambda b|\textbf{1}_{\{0< |a+\lambda b|< \vartheta \}} \goto 0$ as $\vartheta \goto 0$ for almost every 
$(s,x)$ and almost surely. We apply dominated convergence theorem, along with the assumption \ref{A4} to conclude that 
$\mathcal{B}\goto 0$ as $\vartheta \goto 0$. Putting things together and passing to the limit in \eqref{inq:sum-a-b-pathwise-uniqueness}, we have 
\begin{align}
 \mathbb{E}\Big[\int_{D}\big|u_1(t,x)-u_2(t,x)\big|\,{\rm d}x\Big]=0. \notag
\end{align}
In other words, $\mathbb{P}$-a.s.,  $u_1(t,x)=u_2(t,x)$ for almost every $(t,x)$. This yields the uniqueness of path-wise weak solution of the underlying problem \eqref{eq:p-laplace}
with respect to the same stochastic basis. This completes the proof of Theorem \ref{thm:existence-weak}.
\begin{rem}
 Existence of weak solution and path-wise uniqueness guarantee uniqueness in law due to classical Yamada-Watanabe technique \cite{Watanabe} for finite dimensional case; for the infinite dimensional case, see 
 e.g., \cite[Theorems $2$ and $11$]{Ondrejat}.
\end{rem}
\begin{rem}
 Thanks to Skorokhod parameterization, see e.g.,\cite{Dubins, Fernique}, one can prove the following theorem as a generalization of Theorem \ref{thm:existence-weak}: \\
 {\em Let the assumptions \ref{A1}-\ref{A4} be true and $\big(\Omega, \mathcal{F},  \mathbb{P}, \{\mathcal{F}_t\}\big)$ be a given
filtered probability space satisfying the usual hypotheses. Let $N$ be a time-homogeneous Poisson random measure on $\R$ with intensity measure $m(dz)$
defined on $\big(\Omega, \mathcal{F}, \mathbb{P}, \{\mathcal{F}_t\}\big)$, and $\mu$ be a probability measure on $L^2(D)$ such that $\displaystyle \int_{L^2(D)} \Phi(v)\,\mu({\rm d}v) < \infty$ where
$\Phi(v)= \|v\|_{W^{1,p}(D)}^p$. Then, for the problem \eqref{eq:p-laplace}, there exists a weak solution $\hat{\pi}=\big( \hat{\Omega}, \hat{\mathcal{F}},\hat{\mathbb{P}}, \{\hat{\mathcal{F}}_t\},  
 \hat{N}, \hat{u}, \hat{U}\big)$ 
in the sense of Definition \ref{defi:weak-solun} such that \eqref{esti:bound-weak-solun} holds and $\mu=\mathcal{L}(\hat{U})$ on $L^2(D)$.}
\end{rem}

\section{Existence of optimal control: proof of Theorem \ref{thm:existence-optimal-control}}
The objective of this section is to prove existence of a weak optimal solution of \eqref{eq:control-problem} in the sense of Definition \ref{defi:optimal-control} i.e., Theorem \ref{thm:existence-optimal-control}.

\begin{proof} We prove Theorem \ref{thm:existence-optimal-control} in several steps.
\vspace{.1cm}

\noindent{\bf Step I):}
In view of Theorem \ref{thm:existence-weak}, there exists a weak solution of \eqref{eq:p-laplace} with $U=0$, and satisfies the estimate ${\rm ii)}$ of Theorem \ref{thm:existence-weak}. Since $\Psi$ is 
Lipschitz continuous and $ u_{\rm tar}\in L^p(0,T; W_0^{1,p}(D))$, $\Lambda$ is finite. Thus, there exists a minimizing sequence of weak admissible solutions \\
$\pi_n=\big( \Omega_n,  \mathcal{F}_n,  \mathbb{P}_n,\mathbb{F}_n=\{ \mathcal{F}_t^n\}, N_n, u_n, U_n\big)$ such that $\Lambda= \displaystyle \lim_{n\goto \infty}\mathcal{J}(\pi_n)$. Since for
each $n\in \mathbb{N}$, $\pi_n \in \mathcal{U}_{\rm  ad}^w(u_0;T)$, we have, $\mathbb{P}_n$-a.s. for all $t\in [0,T]$
\begin{align}
 u_n(t) &= u_0 + U_n + \int_0^t {\rm div}_x\big( |\nabla u_n|^{p-2} \nabla u_n + \vec{f}(u_n)\big)\,{\rm d}s + \int_0^t \int_{|z|>0} \eta(u_n;z)\widetilde{N}_n({\rm d}z,{\rm d}s) \notag \\
 &= u_0 + U_n + T_{1,n}(t) + T_{2,n}(t)\,. \label{eq:p-laplace-n}
\end{align}
Moreover, since $\Lambda $ is finite, one has the following estimates (uniform in $n$):
\begin{align}\label{esti:apriori-solun-n}
\begin{cases}
\displaystyle  \sup_{n}\mathbb{E}_n\big[ \|U_n\|_{W^{1,p}(D)}^p\big]\le C\,, \\
 \displaystyle \sup_{n}\mathbb{E}_n \Big[\sup_{0\le t\le T} \|u_n(t)\|_{L^2(D)}^2 + \int_0^T \|u_n(t)\|_{W_0^{1,p}(D)}^p\,{\rm d}t \Big] \le C\,,
 \end{cases}
\end{align}
where $\mathbb{E}_n$ denotes the expectation with respect to $\mathbb{P}_n$. 
\vspace{.1cm}

\noindent{\bf Step II):} By proving Aldous condition for the sequence $\{ u_n\}$ in $W^{-1,p^\prime}(D)$ and then applying Theorem \ref{thm:for-tightness} along with the
uniform-estimate \eqref{esti:apriori-solun-n}, one can establish the tightness of $\{ \mathcal{L}(u_n)\}$ on $(\mathcal{Z},\mathcal{T})$. Moreover, due to the uniform-bound \eqref{esti:apriori-solun-n}, and 
the tightness of the family of laws $\{ \mathcal{L}(N_n({\rm d}z,{\rm d}t))\}$ on $M_{\bar{\mathbb{N}}}(\R \times [0,T])$, the set $\{ \mathcal{L}(u_n, N_n, U_n)\}$ is tight in $\mathcal{X}$. Therefore, by 
 \cite[Corollary 2]{motyl2013}, there exist a subsequence of $\{n\}$, still
we denote it by same $\{n\}$, a probability space $(\Omega^*, \mathcal{F}^*, \mathbb{P}^*)$ and, on this space $\mathcal{X}$-valued 
random variables $(u^*, U^*, N^*)$ and $(u_n^*, U_n^*, N_n^*)$ such that 
\begin{itemize}
 \item [i).] $\mathcal{L}(u_n^*, U_n^*, N_n^*)=\mathcal{L}(u_n, U_n, N_n)$ for all $n\in \mathbb{N}$,
 \item [ii).] $(u_n^*, U_n^*, N_n^*)\goto (u^*,U^*, N^*)$ in $\mathcal{X}\quad \mathbb{P}^*$-a.s. ~~~~~$(n\goto \infty)$, 
 \item [iii).] $N_n^*(\omega^*)= N*({\omega}^*)$ for all ${\omega}^*\in {\Omega}^*$.
\end{itemize}
The sequences $\{u_n^*\}$ and $\{ U_n^*\}$ satisfy the same estimate as the original sequences $\{ u_n\}$ and $\{ U_n\}$ respectively. In particular,
\begin{align}\label{esti:apriori-solun-change-variable-n}
\begin{cases}
\displaystyle  \sup_{n}\mathbb{E}^*\big[ \|U_n^*\|_{W^{1,p}(D)}^p\big]\le C\,, \\
\displaystyle \sup_{n}\mathbb{E}^*\Big[\sup_{0\le t\le T} \|u_n^*(t)\|_{L^2(D)}^2 + \int_0^T \|u_n^*(t)\|_{W_0^{1,p}(D)}^p\,{\rm d}t \Big] \le C\,.
\end{cases}
\end{align}
Moreover, in view of \eqref{eq:p-laplace-n} and ${\rm i)}$ of {\bf Step II}, one can conclude $\mathbb{P}^*$-a.s., 
\begin{align}
 u_n^*(t) &= u_0 + U_n^* + \int_0^t {\rm div}_x\big( |\nabla u_n^*(s)|^{p-2} \nabla u_n^*(s) + \vec{f}(u_n^*(s))\big)\,{\rm d}s \notag \\
 & \hspace{3cm}+ \int_0^t \int_{|z|>0} \eta(u_n^*(s);z)\widetilde{N}_n^*({\rm d}z,{\rm d}s)
 \,. \label{eq:p-laplace-change-variable-n}
\end{align}
\vspace{.1cm}

\noindent{\bf Step III):} Let $\mathbb{F}^*$ be the natural filtration of $(u_n^*, N_n^*, u^*, N^*)$. Since  $N_n^*(\omega^*)= N^*({\omega}^*)$ for all ${\omega}^*\in {\Omega}^*$, $N_n^*$ and $N^*$ are the time 
homogeneous Poisson random measures on $\mathbb{R}$ over the stochastic basis $(\Omega^*, \mathcal{F}^*, \mathbb{P}^*, \mathbb{F}^*)$. Using the similar arguments as in Lemmas \ref{lem:convergence-1}- 
\ref{lem:conv-martingale-1} and \ref{lem:convergence-allterm} along with {\bf step ${\rm ii)}$} in subsection \ref{subsec:existence}, one can pass to the limit in \eqref{eq:p-laplace-change-variable-n}
and conclude that the $W_0^{1,p}(D)$-valued $\mathbb{F}^*$-predictable stochastic process $u^*$ satisfies the following:
$\mathbb{P}^*$-a.s.~and a.e.~$t\in [0,T]$,

\begin{align}
 \big( u^*(t), \phi \big)_{L^2(D)}= \big( u_0 + U^*, \phi \big)_{L^2(D)} +  \int_0^t \Big\langle \mbox{div}_x \big(|\nabla u^*(s)|^{p-2}\nabla u^*(s) + \vec{f}(u^*(s))\big), \phi\Big\rangle \,{\rm d}s \notag  \\
  + \Big\langle \int_0^t \int_{|z|>0} \eta(u^*(s);z) \widetilde{N}^*({\rm d}z,{\rm d}s), \phi \Big\rangle \quad \forall~\phi \in W_0^{1,p}(D)\,. \label{eq:p-laplace-new-optimal}
\end{align}
Since $u^* \in \mathbb{D}([0,T]; L_w^2(D))$, \eqref{eq:p-laplace-new-optimal} holds true for all $t\in [0,T]$ and all $\phi \in W_0^{1,p}(D)$, and hence $\pi^*=(\Omega^*, \mathcal{F}^*,
\mathbb{P}^*, \mathbb{F}^*, N^*, u^*, U^*) \in \mathcal{U}_{\rm ad}^w(u_0;T)$. Moreover,  $(u^*, U^*)$ satisfies the estimate 
\eqref{esti:bound-weak-solun}.
\vspace{.1cm}

\noindent{\bf Step IV):} Since $\pi^* \in \mathcal{U}_{\rm ad}^w(u_0;T)$, obviously $\Lambda \le \mathcal{J}(\pi^*)$. We now show that $\mathcal{J}(\pi^*)\le \Lambda$. Note that, the mapping
\begin{align*}
 S: L^2(D)\times W^{1,p}(D) &\mapsto [0,\infty] \\
 (u, U) &\mapsto \|u-u_{\rm tar}\|_{L^2(D)}^2 + \|U\|_{W^{1,p}(D)}^p 
\end{align*}
is a measurable, non-negative and lower semi-continuous convex function. Thus, invoking ${\rm i)}$-${\rm iii)}$ of {\bf step II} along with Fatou's lemma, we get
\begin{align*}
 \mathcal{J}(\pi^*)&=\mathbb{E}^*\Big[ \int_0^T \|u^*(t)-u_{\rm tar}(t)\|_{L^2(D)}^2\,{\rm d}t + \|U^*\|_{W^{1,p}(D)}^p\Big] + \mathbb{E}^*\big[\Psi(u^*(T))\big] \\
 & \le \liminf_{n\goto \infty} \Bigg\{ \mathbb{E}^*\Big[ \int_0^T \|u_n^*(t)-u_{\rm tar}(t)\|_{L^2(D)}^2\,{\rm d}t + \|U_n^*\|_{W^{1,p}(D)}^p\Big] + \mathbb{E}^*\big[\Psi(u_n^*(T))\big] \Bigg\} \\
 &= \liminf_{n\goto \infty} \Bigg\{ \mathbb{E}_n\Big[ \int_0^T \|u_n(t)-u_{\rm tar}(t)\|_{L^2(D)}^2\,{\rm d}t + \|U_n\|_{W^{1,p}(D)}^p\Big] + \mathbb{E}_n\big[\Psi(u_n(T))\big] \Bigg\}\\
 &= \liminf_{n\goto \infty} \mathcal{J}(\pi_n)=\Lambda\,.
\end{align*}
This implies that $\pi^*=(\Omega^*, \mathcal{F}^*,\mathbb{P}^*, \mathbb{F}^*, N^*, u^*, U^*)$ is a weak optimal solution of the control problem \eqref{eq:control-problem}, and $U^*$ is an optimal control.
This finishes the proof of Theorem \ref{thm:existence-optimal-control}.
\end{proof}
\vspace{2cm}

\noindent{\bf Acknowledgements:} The author would like to acknowledge the financial support by Department of Science and Technology, Govt. of India-the INSPIRE fellowship~(IFA18-MA119).

\end{document}